\documentclass[12pt]{amsart}

\usepackage[margin=1.20in]{geometry}

\usepackage{amssymb,amsmath}
\usepackage{amsthm,enumitem}
\usepackage{hyperref}
\usepackage{mathrsfs}
\usepackage{textcomp}
\usepackage{comment}
\usepackage{xcolor}
\hypersetup{
    colorlinks,%
    citecolor=black,%
    filecolor=black,%
    linkcolor=black,%
    urlcolor=black
}

\newtheorem{thm}{Theorem}[section]
\newtheorem{lem}[thm]{Lemma}

\newtheorem{cor}[thm]{Corollary}
\newtheorem{rem}[thm]{Remark}

\newtheorem*{op}{Question}

\numberwithin{equation}{section}

\def\XXint#1#2#3{{\setbox0=\hbox{$#1{#2#3}{\int}$}
\vcenter{\hbox{$#2#3$}}\kern-.5\wd0}}

           \newcommand{\ud}{\mathrm{d}}
\newcommand{\be}{\begin{equation}}      \newcommand{\ee}{\end{equation}}

\newcommand{\R}{\mathbb{R}}

\begin{document}

\title[Fast Diffusion Equations]{\textbf{Optimal Weighted Smoothing and Asymptotics of Ancient Solutions for Fast Diffusion Equations}}

\author{Beomjun Choi, Xiqin Jiang, Hua-Yang Wang, Jingang Xiong}

\address{Department of Mathematical Sciences, KAIST, Daejeon, Korea}
\email{\href{mailto:bchoi@kaist.ac.kr}{\texttt{bchoi@kaist.ac.kr}}}

\address{School of Mathematical Sciences, Laboratory of Mathematics and Complex Systems, MOE, Beijing Normal University, Beijing, China}

\address{Center for Basic Mathematics, Institute for Advanced Study, Beijing Normal University, Beijing, China}

\email{\href{mailto:jiang\_xq@mail.bnu.edu.cn}{\texttt{jiang\_xq@mail.bnu.edu.cn}}}
\email{\href{mailto:wanghuayang@amss.ac.cn}{\texttt{wanghuayang@amss.ac.cn}}}
\email{\href{mailto:jx@bnu.edu.cn}{\texttt{jx@bnu.edu.cn}}}

\thanks{J. Xiong was partially supported by NSFC grant 12325104.}


\begin{abstract}
\setlength{\emergencystretch}{1em}
The Cauchy--Dirichlet problem for the fast diffusion equation on a smooth bounded domain admits a natural bound in the weighted space $L^p_{\Phi_1}$, where $\Phi_1$ is the first Dirichlet eigenfunction. A key regularization question is whether this implies the stronger $L^\infty$ bound.
We provide a complete resolution, showing that the critical exponent coincides with the classical Brezis--Turner exponent from semilinear elliptic theory.

As a primary application, we derive improved global Harnack inequalities and describe asymptotic behavior of positive ancient solutions.
\end{abstract}

\maketitle

\section{Introduction}\label{sec:Introduction}

Let $T>0$, $p\in(1,\infty)$, and let $\Omega \subset \mathbb{R}^n$ ($n \geq 1$) be a smooth bounded domain. 
Let $x=(x_1,\dots,x_n)$ denote points in $\mathbb{R}^n$, and let $t$ denote the time variable. 
Here and below, $a_+:=\max\{a,0\}$, and any quotient with denominator $a_+$ is understood to be $+\infty$ when $a_+=0$.
We consider the Cauchy-Dirichlet problem for the fast diffusion equation:
\begingroup
\renewcommand{\theHequation}{CDP}
\begin{equation}\tag{\bf{CDP}}\label{eq:Cauchy-Dirichlet}
\begin{cases}
\partial_t u^p = \Delta u & \text{in } \Omega \times (0,T),\\
u = 0 & \text{on } \partial\Omega \times (0,T),\\
u(\cdot,0) = u_0 & \text{in } \Omega,
\end{cases}
\end{equation}
\endgroup
where $\Delta=\sum_{i=1}^n \frac{\partial^2}{\partial x_i^2}$ is the Laplace operator and $u_0 \geq 0$ is not identically zero. 
For background and mathematical studies of this problem, we refer to
Chapter 4 of Daskalopoulos--Kenig \cite{DK07}, Chapters 5--7 of
V\'azquez \cite{V07}, and the recent survey of Bonforte--Figalli
\cite{BF24}.

{Let $\Phi_1$ be the positive Dirichlet eigenfunction of $-\Delta$ in $\Omega$ associated with the first eigenvalue $\lambda_1>0$, normalized by $\|\Phi_1\|_{L^2(\Omega)}=1$.}
For $q\in(0,\infty)$, we define the weighted Lebesgue space
\[
L^q_{\Phi_1}(\Omega) := \Big\{ f \in L^q_{\mathrm{loc}}(\Omega) : 
\|f\|_{L^q_{\Phi_1}(\Omega)} := \Big( \int_{\Omega} |f|^q \, \Phi_1 \, \mathrm{d}x \Big)^{1/q} < \infty \Big\}.
\] 
Let $u_0\in L^p_{\Phi_1}(\Omega)$ and set
$u_{0,j}:=\min\{u_0,j\}$.  Denote by $u_j$ the weak solution to
\eqref{eq:Cauchy-Dirichlet} with initial datum $u_{0,j}$.  Since
$u_{0,j}\leq u_{0,j+1}$, it follows from the comparison principle that
$u_j\leq u_{j+1}$.  The monotone limit
\[
u(x,t):=\lim_{j\to\infty}u_j(x,t)
\quad\text{for a.e. }(x,t)\in\Omega\times(0,T),
\]
is called the \emph{limit solution} associated with $u_0$.  V\'azquez
\cite{V07} proved that this limit is well defined, independent of the chosen
monotone bounded approximation, and unique in the class of limit solutions;
see Section \ref{sec:Preliminaries} for the precise statement.
Such solutions exist globally in time but vanish in finite time. In what follows, we always assume $T=\infty$ and let $T^*>0$ denote the extinction time.
A natural question is to study the regularity of limit solutions.  
The most basic problem is whether they enjoy a weighted $L^p_{\Phi_1}$-$L^\infty$ smoothing effect.

A relatively recent development in this field is the notion of weak dual solutions, whose definition relies on the Green representation. 
The existence and uniqueness of minimal weak dual solutions have been established for generalized porous medium type equations; 
see for instance Bonforte-V\'azquez \cite{BV15}, Bonforte-Ibarrondo-Ispizua \cite{BII23}, and Bonforte-Figalli \cite{BF24}.

The global smoothing effects for \eqref{eq:Cauchy-Dirichlet} with  initial data belonging to the unweighted space $ L^q(\Omega)$ 
have been extensively studied. Specifically, it is known that there exists a constant $C = C(n,p,r) > 0$ 
such that for all $t > 0$,
\begin{equation}\label{eq:global_smoothing_effects}
\| u(\cdot, t) \|_{L^{\infty}(\Omega)} 
\leq C \, t^{-n\theta_r} \, \| u_0 \|_{L^r(\Omega)}^{2r\theta_r}, 
\quad \theta_r := \frac{1}{2r - n(p-1)},
\end{equation}
provided that
\begingroup
\renewcommand{\theHequation}{R1}
\begin{equation}\label{eq:initial_datum_range_1}\tag{\textbf{R1}}
\begin{split}
    r \geq p \quad &\text{if} \quad p \in \Big(1, \frac{n}{(n-2)_+}\Big),\\
r > \frac{n(p-1)}{2} \quad &\text{if} \quad p \in \Big[\frac{n}{(n-2)_+}, \infty \Big).
\end{split}
\end{equation}
\endgroup
See Herrero-Pierre \cite{HP85}, DiBenedetto \cite{D93}, Bonforte-Grillo \cite{BG06}, Bonforte-V\'azquez \cite{BV10}, DiBenedetto-Gianazza-Vespri \cite{DGV12}, and Bonforte-Ibarrondo-Ispizua \cite{BII23}. 

For $n\geq2$ and $u_0 \in L^p_{\Phi_1}(\Omega)$,
Daskalopoulos--Kenig~\cite{DK07} obtained an $L^\infty$ estimate for
smooth solutions in the restricted range $p<n/(n-1)$ (see
Lemma~4.6.2), yet without establishing any time-decay rate.
Recently, Bonforte-Ibarrondo-Ispizua \cite{BII23} established weighted smoothing effects for minimal weak dual solutions with $u_0 \in L^r_{\Phi_1}(\Omega)$. 
Specifically, for $n \geq 3$, they showed that for every $t > 0$, there exists a constant $C = C(n, r, p, \Omega) > 0$ such that
\begin{equation}\label{eq:wse_BII23}
\| u(\cdot, t) \|_{L^{\infty}(\Omega)} 
\leq C \, t^{-\frac{n}{r - n(p-1)}} 
\Big( \int_{\Omega} u_0^{r} \, \Phi_1 \, \mathrm{d}x \Big)^{\frac{1}{r - n(p-1)}},
\end{equation}
provided that
\[
r \geq p \quad \text{if} \quad p \in \Big(1, \frac{n}{n-1}\Big), 
\quad \text{and} \quad 
r > n(p-1) \quad \text{if} \quad p \in \Big[\frac{n}{n-1}, +\infty\Big).
\]
When $n\geq2$, our first main result removes the previous
restriction within the natural scale: we extend the admissible range
$r=p$ to the sharp threshold $p<(n+1)/(n-1)$, and, for larger $p$,
identify the corresponding lower bound on the weighted exponent $r$.
Consequently, we obtain a sharp weighted smoothing estimate in the
natural scale, which aligns with known results for semilinear elliptic
equations involving boundary singularities.

\begin{thm}\label{thm:weighted-smoothing_1}
Let $n \geq 1$, $1 < p < \infty$, and let $\Omega \subset \mathbb{R}^n$ be a smooth bounded domain.
Suppose that $u$ is a limit solution to \eqref{eq:Cauchy-Dirichlet} with 
$u_0 \in L^{r}_{\Phi_1}(\Omega)$, where $r > 0$ satisfies
\begingroup
\renewcommand{\theHequation}{R2}
\begin{equation}\label{eq:initial_datum_range_2}\tag{\textbf{R2}} 
\begin{split}
    r \geq p \quad &\text{if} \quad p \in \Big(1, \frac{n+1}{(n-1)_+}\Big), \\
r > \frac{(n+1)(p-1)}{2} \quad &\text{if} \quad p \in \Big[\frac{n+1}{(n-1)_+}, +\infty\Big).
\end{split}
\end{equation}
\endgroup
Then we have $u(\cdot, t) \in L^{\infty}(\Omega)$ for all $t>0$.
Moreover, there exists a constant $C = C(n, r, p, \Omega,q) > 0$ such that for all $t > 0$,
\begin{equation}\label{eq:main-smoothing}
\| u(\cdot, t) \|_{L^{\infty}(\Omega)} 
\leq C \, t^{-\frac{q}{r - q(p-1)}} 
\Big( \int_{\Omega} u_0^{r} \, \Phi_1 \, \mathrm{d}x \Big)^{\frac{1}{r - q(p-1)}},
\end{equation}
where $q = \frac{n+1}{2}$ if $n \geq 2$, while for $n=1$ one may choose any fixed $q\in\bigl(1,\frac{r}{p-1}\bigr)$.
\end{thm} 

When $n\geq2$, the Brezis--Turner threshold in the natural case
$r=p$ is sharp; see Section \ref{sec:Weighted_Smoothing_Effects} for
the counterexample.  The critical exponent $\frac{n+1}{n-1}$
originally appeared in the work of Brezis--Turner \cite{BT77},
who studied the regularity of very weak positive solutions to the semilinear elliptic problem:
\begin{equation}\label{eq:elliptic-equation}
\begin{cases}
-\Delta S = S^p & \text{in } \Omega, \\
S = 0 & \text{on } \partial\Omega.
\end{cases}
\end{equation}
This threshold, known as the \emph{Brezis--Turner exponent}, was shown to be sharp in the elliptic setting by Souplet \cite{S05} and Del Pino-Musso-Pacard \cite{DMP07} 
(see also Quittner-Souplet \cite{QP04} and Bidaut-V\'eron--Ponce--V\'eron \cite{BPV} for related results). 
Our main theorem reveals another subtle connection between the fast diffusion and elliptic theories, as the same critical exponent emerges naturally in the weighted parabolic setting.

As an application of Theorem \ref{thm:weighted-smoothing_1}, we improve the global Harnack inequalities; 
see Theorem \ref{thm:bhi_gengral} and Theorem \ref{thm:bhi_special} below.  Once the $L^\infty$ bound is established, it follows from the optimal regularity results in \cite{JX23,JX25} that
\begin{equation}\label{eq:reg-jx}
\partial_t^l u \in 
\begin{cases}
C^{2+p}(\overline{\Omega} \times (0,T^*)) & \text{if } p \text{ is not an integer}, \\
C^{\infty}(\overline{\Omega} \times (0,T^*)) & \text{if } p \text{ is an integer},
\end{cases}
\quad \forall\, l \geq 0.
\end{equation}
The $C^2$ regularity up to the boundary was conjectured by Berryman-Holland \cite{BH80}. 
Consequently, the limit solutions are in fact classical solutions.

Next, we address the asymptotic behavior of ancient solutions to the fast diffusion equation with zero Dirichlet boundary conditions:
\begingroup
\renewcommand{\theHequation}{ADP}
\begin{equation}\tag{\bf{ADP}}\label{eq:ADP}
\begin{cases}
\partial_t u^p = \Delta u & \text{in } \Omega \times (-\infty, 0),\\
u = 0 & \text{on } \partial \Omega \times (-\infty, 0).
\end{cases}
\end{equation}
\endgroup
Here, an \emph{ancient solution} is a solution $u(x,t)$ defined for all $t \in (-\infty, 0)$ that vanishes at the extinction time $T^*=0$.
Based on the aforementioned  global regularity, we assume throughout that $u$ is a classical solution.
We introduce the following transformation:
\begin{equation}\label{eq:v_transform}
v(x,s) := \Big[ \frac{p}{-(p-1)t} \Big]^{\frac{1}{p-1}} u(x,t), \qquad
s := \frac{p}{p-1} \ln \Big( \frac{1}{-t} \Big), \quad t < 0.
\end{equation}
Under this transformation, the equation \eqref{eq:ADP} becomes
\begin{equation}\label{eq:Equation_of_v}
\begin{cases}
\partial_s v^p = \Delta v + v^p & \text{in } \Omega \times \mathbb{R}, \\
v(x,s) = 0 & \text{on } \partial \Omega \times \mathbb{R}.
\end{cases}
\end{equation}
Let $S$ be a positive solution to \eqref{eq:elliptic-equation}, and define the linearized operator $\mathcal{L}_S$ by
\begin{equation*}
\mathcal{L}_{S} := -\Delta - p S^{p-1}.
\end{equation*}
Following Remark 1.11 in Adams--Simon \cite{AS88}, we say that $S$ is \emph{integrable} if for every nonzero $\phi \in \mathrm{Ker}\,\mathcal{L}_S:=\{\phi \in H^1_0(\Omega): \mathcal{L}_S \phi = 0\}$,
there exists a family $\{S_{\tau}\}_{\tau \in (-1, 1)}$ of solutions to \eqref{eq:elliptic-equation} such that $S_{\tau} \rightarrow S$ in $C^2(\overline{\Omega})$ and $(S_{\tau}-S)/\tau \rightarrow \phi$ in $L^2(\Omega)$ as $\tau \to 0$.
We call $S$ \emph{nondegenerate} if $\mathrm{Ker}\,\mathcal L_S =\{0\}$. Every nondegenerate solution is integrable.
The Morse index of \(S\) is defined as the Morse index of the linearized operator \(\mathcal L_S\), namely, the maximal dimension of a subspace of
$H_0^1(\Omega)$ on which the quadratic form
\[
\mathcal Q_S(\phi):=\int_\Omega
\bigl(|\nabla\phi|^2-pS^{p-1}\phi^2\bigr)\,\ud x
\]
is negative definite.
Now we shall state our second main result, which describes the asymptotic behavior of ancient solutions to \eqref{eq:ADP} as $t \to -\infty$.
\begin{thm}\label{thm:asymptotic_behaviour}
Let $n \geq 1$, $p \in (1, \infty)$, and let $\Omega \subset \mathbb{R}^n$ be a smooth bounded domain.
Suppose that $u \in C^{2}(\overline{\Omega} \times (-\infty, 0))$ is a positive ancient solution to \eqref{eq:ADP} satisfying one of the following conditions:
\begin{itemize}
\item[(i)] $1 < p < \frac{n+1}{(n-1)_{+}}$.
\item[(ii)] There exists $r$ satisfying \eqref{eq:initial_datum_range_1} such that
\[
\limsup_{t \to -\infty}\,(-t)^{-\frac{1}{p-1}} \|u(\cdot,t)\|_{L^{r}(\Omega)} < \infty.
\]
\item[(iii)] There exists $r$ satisfying \eqref{eq:initial_datum_range_2} such that
\[
\limsup_{t \to -\infty}\,(-t)^{-\frac{1}{p-1}} \|u(\cdot,t)\|_{L^{r}_{\Phi_1}(\Omega)} < \infty.
\]
\end{itemize}
{Then there exist a positive solution $v_{-\infty}$ to the elliptic problem \eqref{eq:elliptic-equation} and positive constants $C$ and $\gamma$ such that one of the following alternatives holds:}
\begin{equation}\label{eq:poly_decay}
C^{-1}(-s)^{-1} \leq \Big\| \frac{v(\cdot, s)}{ v_{-\infty}}-1\Big\|_{C^2(\overline{\Omega})} \leq C (-s)^{-\gamma}, \quad \forall s < -1
\end{equation}
or
\begin{equation}\label{eq:exp_decay}
\Big\| \frac{v(\cdot, s)}{ v_{-\infty}}-1\Big\|_{C^2(\overline{\Omega})} \leq C e^{\gamma s}, \quad \forall s < -1
\end{equation}
{Here, $\gamma$ depends only on $n$, $p$, $\Omega$, and $v_{-\infty}$, whereas $C$ is allowed to depend additionally on the ancient solution $u$.}
If $v_{-\infty}$ is integrable, then \eqref{eq:exp_decay} must hold.
\end{thm}

{
\begin{rem}
The dependence of \(C\) on \(u\) reflects the residual scaling
symmetry.  Indeed, \(u_A(x,t):=A u(x,A^{1-p}t)\) has the same extinction
time, while its rescaled trajectory satisfies
\(v_A(x,s)=v(x,s+p\log A)\).  Thus, fixing the extinction time does not
fix the rescaled time phase, and \(C\) cannot in general be chosen
uniformly over the resulting unnormalized class.
\end{rem}
}

It also follows from the preceding asymptotic dichotomy that, under suitable additional conditions, convergence to a stationary backward profile forces the entire rescaled ancient trajectory to be stationary.
\begin{thm}\label{thm:ancient-rigidity}
Let $u$ satisfy the assumptions of Theorem \ref{thm:asymptotic_behaviour}, and let $v_{-\infty}$ be its backward limit.  Suppose in addition that one of the following conditions holds:
\begin{itemize}
\item[(i)] $v_{-\infty}$ is integrable and has Morse index one;
\item[(ii)] $p \in \big(1, \frac{n+2}{(n-2)_+}\big)$ and all positive solutions to \eqref{eq:elliptic-equation} have the same energy (see \eqref{def:energy-functional}).
\end{itemize}
Then $u$ is a separable solution of the form
\begin{equation*}
u(x, t)=\Big[\frac{-(p-1) t}{p}\Big]^{\frac{1}{p-1}} v_{-\infty}(x).
\end{equation*}
\end{thm}
In particular, condition \textup{(i)} applies whenever $v_{-\infty}$ is nondegenerate and has Morse index one. The Morse-index-one assumption in condition \textup{(i)} has a natural
dynamical interpretation. Every positive stationary profile \(S\) has Morse
index at least one, since
\[
\mathcal Q_S(S)
=(1-p)\int_\Omega S^{p+1}\,\ud x<0.
\]
In the rescaled equation, this direction is the constant relative mode induced
by translating the extinction time. Thus, Morse index one means that the extinction-time translation mode is the only linearly unstable direction. In this sense, condition \textup{(i)} expresses linear stability modulo extinction-time translations.
For $p \in \big(1, \frac{n+2}{(n-2)_+}\big)$, the asymptotic behavior of solutions to \eqref{eq:Cauchy-Dirichlet} near the extinction time is well understood; see, among others, Berryman-Holland \cite{BH80}, Feireisl-Simondon \cite{FS00}, Bonforte-Grillo-V\'azquez \cite{BGV12}, Bonforte-Figalli \cite{BF21}, and Akagi \cite{A23}, as well as the results in \cite{CMS23, JX23, JX23a}. Some of these results remain valid for solutions that change sign; see Akagi \cite{A23} and Akagi-Maekawa \cite{A26}.  Our asymptotic analysis for ancient solutions is in the same spirit.

By the classical result of Pohozaev \cite{P65}, there exists no nontrivial solution to \eqref{eq:elliptic-equation} when $\Omega$ is star-shaped and $p \geq \frac{n+2}{(n-2)_+}$. 
Consequently, in this regime, the only separable solution is the trivial one.
On the other hand, when $n\geq3$, Bahri-Coron \cite{BC88}
proved the existence of positive solutions to
\eqref{eq:elliptic-equation} when $p=\frac{n+2}{n-2}$ and $\Omega$
has nontrivial topology.  For
$\frac{n+1}{n-1}<p<\frac{n+2}{n-2}$, we expect the existence of
self-similar solutions in the half-space
$\Omega=\mathbb{R}^n_+$.
Moreover, in bounded domains, the rescaled solution $(-t)^{-\frac{1}{p-1}} u(x,t)$ may exhibit self-similar type blow-up as $t \to -\infty$. The Sobolev critical case $p=\frac{n+2}{n-2}$ is expected to exhibit richer 
phenomena; see Daskalopoulos-del Pino-Sesum \cite{DDS} for the problem in the whole space and the study of extinction profiles in \cite{JX26}.

Finally, we note that  weighted spaces $L^q_{\Phi_1}(\Omega)$ have played crucial roles in the theory of nonlinear elliptic and parabolic Dirichlet problems.
Their significance for semilinear problems was highlighted by Brezis-Cazenave-Martel-Ramiandrisoa \cite{BCMR96}, 
who established a link between the existence of global classical solutions to parabolic equations and the existence of very weak solutions to their stationary counterparts.
These spaces are instrumental in studying complete or instantaneous blow-up phenomena; see Brezis-Cabr\'e \cite{BC98}, Cabr\'e-Martel \cite{CM98, CM99}, and Martel \cite{M99}.
For smoothing properties in weighted spaces for semilinear parabolic or elliptic  equations, see Fila-Souplet-Weissler \cite{FSW01} and Quittner-Souplet \cite{QP04}. Their proofs were based on a novel bootstrap argument, relying on regularity theory for linear elliptic and parabolic equations.
Our proof of Theorem \ref{thm:weighted-smoothing_1} proceeds differently. 
Specifically, we establish a weighted Caccioppoli-type inequality 
(Lemma \ref{lem:energy_estimate}) and implement a modified Moser iteration. 

The paper is organized as follows. 
Section \ref{sec:Preliminaries} introduces preliminary tools, including weighted Sobolev inequalities. 
In Section \ref{sec:Weighted_Smoothing_Effects}, we prove Theorem \ref{thm:weighted-smoothing_1} using a variant of the Moser iteration method. 
Section \ref{sec:Harnack} is devoted to establishing global Harnack inequalities via an almost representation formula and our sharp smoothing estimates. 
Finally, Section \ref{sec:Ancient} derives basic a priori estimates for ancient solutions, establishes their asymptotic behavior as $t \to -\infty$ in Theorem \ref{thm:asymptotic_behaviour}, and proves the rigidity criteria in Theorem \ref{thm:ancient-rigidity}.

\section{Preliminaries}\label{sec:Preliminaries}

As mentioned previously, $\Phi_1$ denotes the positive eigenfunction corresponding to the first Dirichlet eigenvalue $\lambda_1>0$, normalized by $\|\Phi_1\|_{L^2(\Omega)}=1$, and satisfies
\begin{equation}\label{eq:first-eigenfunction}
\begin{cases}
-\Delta \Phi_1=\lambda_1 \Phi_1 & \text{in }\Omega,\\
\Phi_1=0 & \text{on }\partial \Omega.
\end{cases}
\end{equation}
By the Hopf boundary lemma and the regularity theory of elliptic equations, there exists a constant $C\ge 1$ such that 
\begin{equation}\label{eq:hopf-estimate}
\frac{1}{C}\, d(x) \le \Phi_1(x) \le C\, d(x), \quad \forall x\in\Omega,
\end{equation}
where $d(x):=\text{dist}(x,\partial \Omega)$.
Let $W^{1,2}_{\Phi_1}(\Omega)$ be the weighted Sobolev space consisting of functions $f\in W^{1,2}_{\text{loc}}(\Omega)$ with finite norm
\begin{align}\label{eq:weighted-norm}
\|f\|_{W^{1,2}_{\Phi_1}(\Omega)} := \left( \int_{\Omega} \big(|f|^{2}+|\nabla f|^{2}\big)\, \Phi_1 \,\mathrm{d}x \right)^{1/2}.
\end{align}

Throughout this paper, the exponent $q^*$ is defined by $q^{*} = \frac{2 n+ 2}{n-1}$ for $n\geq2$.  When $n=1$, we fix any $q^*>2$ and define $q>1$ by $q^{-1}+2(q^*)^{-1}=1$.
We recall the following weighted Sobolev inequality.

{
\begin{lem}\label{lem:weighted-Poincare}
Let $n \geq 1$, and let $\Omega\subset \R^n$ be a smooth bounded domain.
Then there exists a constant $C_1=C_1(n,\Omega,q^*)>0$ such that, for every
$f\in W^{1,2}_{\Phi_1}(\Omega)$,
\begin{equation}\label{eq:weighted-Sobolev}
\Big(\int_{\Omega} |f|^{q^{*}} \Phi_1 \,\ud x \Big)^{{2}/{q^{*}}}
\le C_1 \int_{\Omega}\bigl(|\nabla f|^{2}+|f|^{2}\bigr)\Phi_1\,\ud x .
\end{equation}
\end{lem}
\begin{proof}
Define
\[
\widetilde{\Omega}
:=\{(x,x_{n+1})\in\R^{n+1}:x\in\Omega,\ 0<x_{n+1}<\Phi_1(x)\}
\]
and $\widetilde f(x,x_{n+1}):=f(x)$.  By the boundary regularity of
$\Phi_1$ and the Hopf boundary lemma, $\widetilde\Omega$ is a bounded
Lipschitz domain.  By the classical Sobolev embedding
$W^{1,2}(\widetilde\Omega)\hookrightarrow L^{q^*}(\widetilde\Omega)$
and the Fubini--Tonelli theorem, we have
\begin{align*}
\int_{\Omega}|f|^{q^*}\Phi_1\,\ud x
&=\int_{\widetilde\Omega}|\widetilde f|^{q^*}\,\ud x\,\ud x_{n+1}\\
&\le C\left(\int_{\widetilde\Omega}
\bigl(|\nabla\widetilde f|^2+|\widetilde f|^2\bigr)\,\ud x\,\ud x_{n+1}
\right)^{q^*/2}\\
&=C\left(\int_\Omega
\bigl(|\nabla f|^2+|f|^2\bigr)\Phi_1\,\ud x\right)^{q^*/2}.
\end{align*}
\end{proof}

Consequently, we have the following parabolic weighted Sobolev inequality.
\begin{lem}\label{lem:weighted-Sobolev-variant}
Let $\sigma\in(1,q^*/2)$, and let $q\in(1,\infty)$ be determined by
\[
\frac1q+\frac2{q^*}=1.
\]
Assume that $f(\cdot,\tau)\in W^{1,2}_{\Phi_1}(\Omega)$ for a.e.
$\tau\in(s,t)$.  Then
\begin{align*}
\int_s^t\!\int_\Omega |f|^{2\sigma}\Phi_1\,\ud x\,\ud\tau
&\le C_1\left(\int_s^t\!\int_\Omega
\bigl(|\nabla f|^2+|f|^2\bigr)\Phi_1\,\ud x\,\ud\tau\right)\\
&\quad\times
\left(\mathop{\rm ess\,sup}_{\tau\in(s,t)}
\int_\Omega |f(\cdot,\tau)|^{2(\sigma-1)q}\Phi_1\,\ud x
\right)^{1/q}.
\end{align*}
\end{lem}
\begin{proof}
For a.e. $\tau\in(s,t)$, by H\"older's inequality and
\eqref{eq:weighted-Sobolev}, we have
\begin{align*}
\int_\Omega |f|^{2\sigma}\Phi_1\,\ud x
&\le
\left(\int_\Omega |f|^{q^*}\Phi_1\,\ud x\right)^{2/q^*}
\left(\int_\Omega |f|^{2(\sigma-1)q}\Phi_1\,\ud x\right)^{1/q}\\
&\le C_1
\left(\int_\Omega\bigl(|\nabla f|^2+|f|^2\bigr)\Phi_1\,\ud x\right)
\left(\int_\Omega |f|^{2(\sigma-1)q}\Phi_1\,\ud x\right)^{1/q}.
\end{align*}
Integrating in $\tau$ and taking the essential supremum of the second
factor proves the assertion.
\end{proof}
}
We remark that $q = \frac{n+1}{2}$ when $n \geq 2$, whereas for $n=1$ any $q>1$ can be obtained by taking $q^*=2q/(q-1)$.  In applications below we additionally choose $q<r/(p-1)$.
Finally, we recall the following iteration lemma.
\begin{lem}[De Giorgi]\label{lem:iteration}
Let $Z: [a, d] \to [0,\infty)$ be a bounded nonnegative function. 
Assume that there exist constants $A,B\geq 0$, $\alpha>0$, and $\theta \in [0,1)$ such that for any $a\leq b<c<d$,
\begin{equation*}
Z(b) \leq A (c-b)^{-\alpha} + B + \theta Z(c).
\end{equation*}
Then there exists a constant $C_3=C_3(\alpha,\theta)>0$ such that
\begin{equation*}
Z(a) \leq C_3 (A (d-a)^{-\alpha} + B).
\end{equation*}
\end{lem}
\begin{proof}
Choose a constant $\lambda \in (0, 1)$ such that 
$$t_0=a,\quad t_{j+1} - t_j = \lambda^{j}(1 - \lambda)(d-a).$$
We will complete the proof by iterating the inequality; see Lemma 6.3 in
Bonforte--Ibarrondo--Ispizua \cite{BII23} for more details.
\end{proof}

For readers' convenience, we recall the existence and uniqueness of limit solutions to \eqref{eq:Cauchy-Dirichlet} with $u_0 \in L^p_{\Phi_1}(\Omega)$ in our notation; see Theorem 6.15 of V\'azquez \cite{V07}.
\begin{thm}\label{thm:limit_sol}
For every nonnegative $u_0 \in L^p_{\Phi_1}(\Omega)$, there exists a unique limit solution $u \in C([0, \infty); L^p_{\Phi_1}(\Omega))$ to the Cauchy-Dirichlet problem \eqref{eq:Cauchy-Dirichlet}.
Furthermore, if $u$ and $v$ are two limit solutions to \eqref{eq:Cauchy-Dirichlet} with nonnegative initial data $u_{0}, v_{0} \in L^p_{\Phi_1}(\Omega)$ respectively, then the following \emph{weighted contraction principle} holds:
\begin{equation*}
\int_{\Omega}(u^p(x, t) - v^p(x, t))_{\pm}\Phi_1\,\ud x
\leq \int_{\Omega}(u_0^p-v_0^p)_{\pm}\Phi_1\,\ud x,
\qquad \forall t>0.
\end{equation*}
\end{thm}

The weighted contraction principle ensures that the limit solution is well-defined, that is, it is independent of the specific choice of approximating sequences.
For weak solutions used below, we shall use the standard formulation: for every $0<t_1<t_2<T$ and every admissible test function $\varphi$ vanishing on $\partial\Omega\times(0,T)$,
\begin{equation}\label{eq:weak-sol}
\begin{split}
  & \int_{t_1}^{t_2}\!\int_{\Omega} \left( u^p \, \partial_{\tau} \varphi - \nabla u \cdot \nabla \varphi \right) 
\, \mathrm{d}x \, \mathrm{d}\tau \\
= & \int_{\Omega} u(x, t_2)^p \, \varphi(x, t_2) \, \mathrm{d}x 
- \int_{\Omega} u(x, t_1)^p \, \varphi(x, t_1) \, \mathrm{d}x .
\end{split}
\end{equation}
It follows from V\'azquez \cite{V07} and Bonforte--V\'azquez \cite{BV15} that every limit solution is a very weak solution in the standard transposition sense.
More precisely, $u^p \in C([0,\infty); L^1_{\Phi_1}(\Omega))$, $u^p(\cdot,0)=u_0^p$ in $L^1_{\Phi_1}(\Omega)$, and
\begin{equation*}
\int_{0}^{\infty}\!\int_{\Omega}u^p \partial_t \psi\,\ud x\,\ud t
+ \int_{0}^{\infty}\!\int_{\Omega}u \Delta \psi\,\ud x\,\ud t=0
\end{equation*}
for every $\psi \in C_c^{\infty}((0,\infty)\times \overline{\Omega})$ satisfying $\psi=0$ on $\partial\Omega\times(0,\infty)$.
In general, limit solutions are merely very weak solutions but not weak solutions; see subsection \ref{subsec:example} for an example.
However, as indicated in the introduction, under our assumptions in Theorem \ref{thm:weighted-smoothing_1}, all such limit solutions are indeed classical solutions.
This distinction naturally raises the following question:
\begin{op}
If $u_0 \in L^p_{\Phi_1}(\Omega)$, can we establish the uniqueness or non-uniqueness of the very weak solution to \eqref{eq:Cauchy-Dirichlet}? 
In other words, does the equivalence between limit solutions and very weak solutions hold in this context?
\end{op}

\section{Weighted Smoothing Effects}\label{sec:Weighted_Smoothing_Effects}
Here and in the sequel, we assume $n \geq 1$, $p>1$, and that $\Omega$ is a smooth bounded domain, unless otherwise indicated.
In this section, we aim to establish weighted smoothing effects for limit solutions to the Cauchy-Dirichlet problem \eqref{eq:Cauchy-Dirichlet}.
We first establish the following energy inequality.
\begin{lem}\label{lem:energy_estimate}
Let $\theta\geq \theta_1>1$ and $0 \leq s_1 < s_2 < s<\infty$. 
Then there exists a constant $C_4=C_4(p, \theta_1)>0$ such that for any weak solution $u$ to \eqref{eq:Cauchy-Dirichlet} with $u_0 \in L^{\infty}(\Omega)$,
{
\begin{align*}
&\sup_{\tau \in (s_2, s]} \int_{\Omega} u(x,\tau)^{p+\theta} \Phi_1 \,\ud x
+ \int_{s_2}^{s}\!\int_{\Omega} |\nabla u^{(\theta+1)/2}|^{2} \Phi_1 \,\ud x\,\ud \tau\\
&\qquad
+\lambda_1\int_{s_2}^{s}\!\int_\Omega u^{\theta+1}\Phi_1\,\ud x\,\ud\tau
\leq \frac{C_4}{s_2 - s_1} \int_{s_1}^{s}\!\int_{\Omega} u^{p+\theta} \Phi_1 \,\ud x \, \ud \tau.
\end{align*}
}
\end{lem}
\begin{proof}
Since $u_0 \in L^{\infty}(\Omega)$, by the global smoothing effects \eqref{eq:global_smoothing_effects}, we have $u(\cdot, t) \in L^{\infty}(\Omega)$ for all $t>0$.
By the regularity results of \cite{JX23, JX25}, we have $u \in C^{2}(\overline{\Omega}\times (0, T^*))$.
Let $\eta \in C^{\infty}([0,\infty))$ satisfy $0\leq\eta\leq1$, $\eta=0$ on $[0,s_1]$, $\eta=1$ on $[s_2,\infty)$, and
\begin{equation*}
|\eta'| \leq \frac{2}{s_2 - s_1}.
\end{equation*}
Choosing $\varphi = \eta^2 u^{\theta}\Phi_1 \in C^1(\overline{\Omega}\times (0,T^*)), \theta>1$ as a test function in \eqref{eq:weak-sol}, we have
\begin{align*}
\int_{s_1}^{s}\!\int_{\Omega} u^p \partial_t (\eta^2 u^{\theta}\Phi_1) \,\ud x \, \ud \tau- \int_{s_1}^{s}\!\int_{\Omega}\nabla u \cdot \nabla (\eta^2 u^{\theta}\Phi_1)\,\ud x \, \ud \tau 
= \int_{\Omega} u(x,s)^{p+\theta} \Phi_1\,\ud x.
\end{align*}
Using integration by parts, we obtain
\begin{align*}
\int_{s_1}^{s}\!\int_{\Omega} u^p \partial_t (\eta^2 u^{\theta}\Phi_1) \,\ud x \, \ud \tau 
= & \int_{s_1}^{s}\!\int_{\Omega} \Big( \frac{\theta }{p}\eta^2 \Phi_1  u^{\theta}\partial_t (u^{p}) + 2\eta \eta' u^{p+\theta} \Phi_1 \Big) \,\ud x \, \ud \tau \\
= & -\frac{\theta }{p}\!\int_{s_1}^{s}\!\int_{\Omega}
\nabla u \cdot \nabla (\eta^2 u^{\theta}\Phi_1)\,\ud x \, \ud \tau
+ \int_{s_1}^{s}\!\int_{\Omega}
2\eta \eta' u^{p+\theta} \Phi_1 \,\ud x \, \ud \tau,
\end{align*}
and
\begin{align*}
\int_{s_1}^{s}\!\int_{\Omega}\nabla u \cdot \nabla (\eta^2 u^{\theta}\Phi_1)\,\ud x \, \ud \tau 
= & \int_{s_1}^{s}\!\int_{\Omega} \left( \theta \eta^2 u^{\theta-1} \Phi_1 |\nabla u|^{2} + \eta^2 u^{\theta} \nabla u \cdot \nabla \Phi_1 \right) \,\ud x \, \ud \tau \\
= & \int_{s_1}^{s}\!\int_{\Omega} \Big( \frac{4\theta}{(\theta+1)^2} \eta^2 \Phi_1 |\nabla u^{\frac{\theta+1}{2}}|^{2}  + \frac{\lambda_1}{1+\theta} \eta^2 u^{\theta+1} \Phi_1 \Big) \,\ud x \, \ud \tau.
\end{align*}
{Note that
\[
\min\left\{1,\frac{4\theta(p+\theta)}{p(\theta+1)^2},
\frac{p+\theta}{p(\theta+1)}\right\}
\geq
\min\left\{1,\frac{4\theta_1^2}{p(\theta_1+1)^2},\frac1p\right\}
=:c_1(p,\theta_1)^{-1}>0.
\]}
Combining the preceding identities, we obtain
{
\begin{align*}
&\int_{\Omega} u(x,s)^{p+\theta} \Phi_1 \,\ud x
+ \int_{s_2}^{s}\!\int_{\Omega} |\nabla u^{(\theta+1)/2}|^{2} \Phi_1 \,\ud x \, \ud \tau\\
&\qquad
+\lambda_1\int_{s_2}^{s}\!\int_\Omega u^{\theta+1}\Phi_1\,\ud x\,\ud\tau
\leq \frac{4c_1}{s_2 - s_1} \int_{s_1}^{s}\!\int_{\Omega} u^{p+\theta} \Phi_1 \,\ud x \, \ud \tau.
\end{align*}
}
Notice that there exists $s_0 \in (s_2, s)$ such that
\begin{align*}
\frac{1}{2}\sup_{\tau \in (s_2, s]} \int_{\Omega} u(x,\tau)^{p+\theta} \Phi_1 \,\ud x 
\leq \int_{\Omega} u(x,s_0)^{p+\theta} \Phi_1 \,\ud x \leq \frac{4c_1}{s_2 - s_1} \int_{s_1}^{s}\!\int_{\Omega} u^{p+\theta} \Phi_1 \,\ud x \, \ud \tau,
\end{align*}
where we have used the last inequality for $s=s_0$.
Combining the above two inequalities, we complete the proof with $C_4 = 12c_1>0$.
\end{proof}

\begin{thm}\label{thm:weighted-smoothing}
Suppose that $u$ is a limit solution to \eqref{eq:Cauchy-Dirichlet} with $u_0 \in L^{r}_{\Phi_1}(\Omega)$, where $r>0$ satisfies \eqref{eq:initial_datum_range_2}.
Then for any $0 \leq t_0 < t < \infty$, there exists a constant $C_5=C_5(n,p,r,\Omega,q)>0$ such that
\begin{equation}\label{equ:L_infty-to-Lr}
\|u(\cdot,t)\|_{L^{\infty}(\Omega)} 
\leq C_5  ( t - t_0  )^{-\frac{q+1}{r - q(p-1)}} \Big( \int_{t_0}^{t}\!\int_{\Omega} u^{r} \Phi_1 \,\ud x \, \ud \tau \Big)^{\frac{1}{r- q(p-1)}},
\end{equation}
where $q = \frac{n+1}{2}$ if $n \geq 2$, while for $n=1$ we fix any $q\in\bigl(1,\frac{r}{p-1}\bigr)$.
\end{thm}
\begin{proof}
We first reduce to bounded initial data.  Let
$u_{0,k}:=\min\{u_0,k\}$ and let $u_k$ be the corresponding solutions.
By the comparison principle, the sequence $\{u_k\}$ is nondecreasing;
by the definition of a limit solution, $u_k\uparrow u$ a.e.
Once the desired estimate is proved for $u_k$, with a constant
independent of $k$, since $u_k\leq u$, we have
\begin{align*}
\|u_k(\cdot,t)\|_{L^\infty(\Omega)}
&\leq C_5(t-t_0)^{-\frac{q+1}{r-q(p-1)}}
 \Big(\int_{t_0}^t\!\int_\Omega u_k^r\Phi_1\,\ud x\,\ud\tau\Big)^{\frac1{r-q(p-1)}}\\
&\leq C_5(t-t_0)^{-\frac{q+1}{r-q(p-1)}}
 \Big(\int_{t_0}^t\!\int_\Omega u^r\Phi_1\,\ud x\,\ud\tau\Big)^{\frac1{r-q(p-1)}}.
\end{align*}
Letting $k\to\infty$ proves \eqref{equ:L_infty-to-Lr}, since
$u_k(\cdot,t)\uparrow u(\cdot,t)$ a.e.  Thus it suffices to consider
$u_0\in L^\infty(\Omega)\cap L^r_{\Phi_1}(\Omega)$.

Let $\sigma \in (1, q^*/2)$, where $1/q + 2/q^* = 1$.
Note that $q^{*} = \frac{2 n+ 2}{n-1}$ for $n \geq 2$, while $q^*=2q/(q-1)>2$ for $n=1$.
{
Set
\[
(p+\theta)(\sigma-1)q=2p+\theta-1,
\qquad
\text{that is,}\qquad
\sigma=1+\frac{2p+\theta-1}{(p+\theta)q}.
\]
By Lemma \ref{lem:weighted-Sobolev-variant}, applied to
$f=u^{(p+\theta)/2}$, we have
\begin{align*}
\int_{s_2}^{t}\!\int_\Omega u^{(p+\theta)\sigma}\Phi_1\,\ud x\,\ud\tau
&\leq C_1
\left(\int_{s_2}^{t}\!\int_\Omega
\bigl(|\nabla u^{(p+\theta)/2}|^2+u^{p+\theta}\bigr)
\Phi_1\,\ud x\,\ud\tau\right)\\
&\quad\times
\left(\mathop{\rm ess\,sup}_{\tau\in(s_2,t]}
\int_\Omega u^{2p+\theta-1}\Phi_1\,\ud x\right)^{1/q}.
\end{align*}
Apply Lemma \ref{lem:energy_estimate} with the parameter
$\vartheta=p+\theta-1$; here $\vartheta\geq p>1$, so the constant is
uniform throughout the iteration.  Since
$\vartheta+1=p+\theta$ and $p+\vartheta=2p+\theta-1$, it follows that
\begin{align*}
&\int_{s_2}^{t}\!\int_\Omega
\bigl(|\nabla u^{(p+\theta)/2}|^2+u^{p+\theta}\bigr)
\Phi_1\,\ud x\,\ud\tau\\
&\qquad\leq
\frac{C_E}{s_2-s_1}
\int_{s_1}^{t}\!\int_\Omega u^{2p+\theta-1}\Phi_1\,\ud x\,\ud\tau,
\qquad
C_E:=C_4\max\{1,\lambda_1^{-1}\},
\end{align*}
and
\[
\mathop{\rm ess\,sup}_{\tau\in(s_2,t]}
\int_\Omega u^{2p+\theta-1}\Phi_1\,\ud x
\leq
\frac{C_4}{s_2-s_1}
\int_{s_1}^{t}\!\int_\Omega u^{2p+\theta-1}\Phi_1\,\ud x\,\ud\tau.
\]
Consequently,
\begin{align}\label{eq:norm_estimate}
\int_{s_2}^{t}\!\int_{\Omega} u^{(p+\theta)\sigma} \Phi_1 \,\ud x \, \ud \tau
\leq \frac{C_M}{(s_2-s_1)^{1+q^{-1}}}
\left(\int_{s_1}^{t}\!\int_\Omega
u^{2p+\theta-1}\Phi_1\,\ud x\,\ud\tau\right)^{1+q^{-1}},
\end{align}
where $C_M:=C_1C_EC_4^{1/q}$.
}
Define 
$$ r_k = 2p + \theta_k -1 = (p+\theta_k)(\sigma_k-1)q, \quad r_{k+1} = (p + \theta_k)\sigma_k,$$
so that
$$ r_{k+1} =(1+q^{-1}) r_k + 1 - p,\quad 
\sigma_k = 1 + \frac{r_k}{(r_k+1-p) q},\quad \theta_k = r_k+1-2p.$$
Consequently, we have the explicit formula 
$$ r_{k+1} = \Big(\frac{q+1}{q}\Big)^k(r_1-q(p-1))+ q(p-1).$$
If we choose 
$$r_1 > \max\big\{2p, q(p-1)\big\},$$
then $\{r_k\}$ and $\{\theta_k\}$ are increasing sequences such that for $k \in \mathbb{Z}_{+}$, 
$$ r_{k+1}>r_k \geq r_1>p,\quad \theta_{k+1} >\theta_k \geq r_1+1-2p >1, \quad \sigma_k \in \Big(1, \frac{q^*}{2}\Big),$$ 
and hence the above estimates are well defined.
Moreover, it follows from \eqref{eq:norm_estimate} that
{
\begin{equation}\label{eq:iteration_norm}
\int_{t_{k+1}}^{t}\!\int_{\Omega} u^{r_{k+1}} \Phi_1 \,\ud x \, \ud \tau  
\leq  \frac{C_M}{(t_{k+1} - t_k)^{1 + q^{-1}}} \Big( \int_{t_{k}}^{t}\!\int_{\Omega} u^{r_{k}} \Phi_1 \,\ud x \, \ud \tau \Big)^{1 + q^{-1}},
\end{equation}
}
where $\{t_k\}$ is a sequence chosen by
$$ t_1\geq 0,\quad t_{\infty} = \lim_{k \to \infty} t_k,\quad t_{k+1} - t_k = \frac{\tau}{k^2}~~\text{with}~~\tau = (t_{\infty} - t_1)\Big(\sum_{k=1}^\infty \frac{1}{k^2}\Big)^{-1}.$$
{Set $\alpha=1+q^{-1}$ and $C=C_M^{1/\alpha}$.}
Thus we could rewrite \eqref{eq:iteration_norm} as
\begin{equation*}
\Big(\int_{t_{k+1}}^{t}\!\int_{\Omega} u^{r_{k+1}} \Phi_1 \,\ud x \, \ud \tau\Big)^{\frac{1}{r_{k+1}}}
\leq  \Big( \frac{C}{\tau} \Big)^{\frac{\alpha}{r_{k+1}}}k^{\frac{2\alpha}{r_{k+1}}} \Big( \int_{t_{k}}^{t}\!\int_{\Omega} u^{r_{k}} \Phi_1 \,\ud x \, \ud \tau \Big)^{\frac{1}{r_k}\frac{\alpha r_k}{r_{k+1}}}.
\end{equation*}
Iterating this inequality, we obtain
\begin{equation*}
\Big(\int_{t_{k+1}}^{t}\!\int_{\Omega} u^{r_{k+1}} \Phi_1 \,\ud x \, \ud \tau\Big)^{\frac{1}{r_{k+1}}}
\leq  \Big( \frac{C}{\tau} \Big)^{\frac{\sum\limits^{k}_{j=1}\,\alpha^j}{r_{k+1}}}
\prod_{j=1}^{k}  j^{\frac{2\alpha^{k-j+1}}{r_{k+1}}} \Big( \int_{t_1}^{t}\!\int_{\Omega} u^{r_{1}} \Phi_1 \,\ud x \, \ud \tau \Big)^{\frac{\alpha^k}{r_{k+1}}}.
\end{equation*}
Notice that 
\begin{align*}
\lim_{k \to \infty}\,\frac{\sum\limits_{j=1}^k\, \alpha^j}{r_{k+1}} = \frac{q+1}{r_1-q(p-1)}, \quad
\lim_{k \to \infty}\,\frac{\alpha^k}{r_{k+1}} = \frac{1}{r_1 - q(p-1)}.
\end{align*}
Besides, there exists a constant $c_2>0$ such that
\begin{align*}
\prod_{j=1}^{k}\,j^{\frac{2\alpha^{k-j+1}}{r_{k+1}}} 
= \exp\Big( \frac{2\alpha}{r_{1}-q(p-1)+q(p-1)\alpha^{-k}} \sum_{j=1}^{k}\, \alpha^{-j} \ln j \Big)
\rightarrow c_2,~~\text{as}~~k\to \infty
\end{align*}
since
$\lim\limits_{k\rightarrow\infty} \frac{\alpha^{-k-1}\ln(k+1)}{\alpha^{-k}\ln(k)} = \alpha^{-1} < 1$.
As a consequence, there is a constant $c_3>0$ such that
\begin{align*}
\|u\|_{L^{\infty}(\Omega \times (t_{\infty}, t])} 
\leq & c_2C^{\frac{q+1}{r_1 - q(p-1)}} \tau^{-\frac{q+1}{r_1 - q(p-1)}} \Big( \int_{t_1}^{t}\!\int_{\Omega} u^{r_{1}} \Phi_1 \,\ud x \, \ud \tau \Big)^{\frac{1}{r_1 - q(p-1)}}\\
= & c_3 \left( t_{\infty} - t_1 \right)^{-\frac{q+1}{r_1 - q(p-1)}} \Big( \int_{t_1}^{t}\!\int_{\Omega} u^{r_{1}} \Phi_1 \,\ud x \, \ud \tau \Big)^{\frac{1}{r_1 - q(p-1)}}.
\end{align*}
If the exponent $r$ in the statement satisfies $r>\max\{2p,q(p-1)\}$, we take $r_1=r$, $t_1=t_0$, and $t_\infty=(t_0+t)/2$ in the preceding estimate.  Since $t\in(t_\infty,t]$, this proves \eqref{equ:L_infty-to-Lr} in this case.

It remains to treat $q(p-1)<r\leq\max\{2p,q(p-1)\}$.  Fix $r_1>\max\{2p,q(p-1)\}$ and set
\[
Z(\tau):=\|u\|_{L^\infty(\Omega\times(\tau,t])}.
\]
For $t_0\leq c<d<t$, by the preceding high-exponent estimate, applied
with $t_1=c$ and $t_\infty=d$, and H\"older's inequality, we have
\begin{align*}
Z(d)
&\leq C(d-c)^{-\frac{q+1}{r_1-q(p-1)}}
Z(c)^{\frac{r_1-r}{r_1-q(p-1)}}
\Big(\int_{t_0}^{t}\!\int_\Omega u^r\Phi_1\,\ud x\,\ud\tau\Big)^{\frac1{r_1-q(p-1)}}.
\end{align*}
By Young's inequality, we have
\begin{align*}
Z(d)
\leq \frac12 Z(c)+C(d-c)^{-\frac{q+1}{r-q(p-1)}}
\Big(\int_{t_0}^{t}\!\int_\Omega u^r\Phi_1\,\ud x\,\ud\tau\Big)^{\frac1{r-q(p-1)}}.
\end{align*}
Let $b=(t_0+t)/2$ and define $\widetilde Z(\rho)=Z(t_0+b-\rho)$ on $[t_0,b]$.  The last inequality has exactly the form required by Lemma \ref{lem:iteration}; hence
\[
Z(b)\leq C(b-t_0)^{-\frac{q+1}{r-q(p-1)}}
\Big(\int_{t_0}^{t}\!\int_\Omega u^r\Phi_1\,\ud x\,\ud\tau\Big)^{\frac1{r-q(p-1)}}.
\]
Since $\|u(\cdot,t)\|_\infty\leq Z(b)$ and $b-t_0=(t-t_0)/2$, this proves \eqref{equ:L_infty-to-Lr}.
This completes the proof.
\end{proof}
We recall the following monotonicity properties of the $L^r_{\Phi_1}(\Omega)$ norms of solutions to \eqref{eq:Cauchy-Dirichlet}.
\begin{lem}\label{lem:monotonicity}
Let $r \geq p$, and let $u$ be a limit solution to \eqref{eq:Cauchy-Dirichlet} with $ u_0 \in L^{r}_{\Phi_1}(\Omega)$.
Then for any $0 \leq t_0 < t < \infty$,
\begin{equation}\label{eq:monotonicity_norm}
\|u(\cdot, t)\|_{L^{r}_{\Phi_1}(\Omega)} \leq \|u(\cdot, t_0)\|_{L^{r}_{\Phi_1}(\Omega)}.
\end{equation}
Moreover, if $r=p$, we have
\begin{equation}\label{eq:monotonicity_norm_p_I}
\|u(\cdot, t)\|_{L^{p}_{\Phi_1}(\Omega)}^{p-1} \geq \|u(\cdot, t_0)\|_{L^{p}_{\Phi_1}(\Omega)}^{p-1} -\frac{\lambda_1(p-1)}{p}\|\Phi_1\|_{L^1(\Omega)}^{(p-1)/p} (t - t_0).
\end{equation}
\end{lem}
\begin{proof}
The computation below is first justified for the bounded
solutions with truncated data $\min\{u_0,j\}$, and then passed to the
limit by monotone convergence.
For any $\tau \in (t_0, t)$, we have
\begin{multline*}
\frac{\ud}{\ud\tau}\int_\Omega u^r\Phi_1\,\ud x
= \frac{r}{p}\int_\Omega u^{r-p}\partial_\tau(u^p)\Phi_1\,\ud x \\
= \frac{r}{p} \int_{\Omega} u^{r-p} \Delta u \Phi_1 \,\ud x \\
= -\frac{r(r-p)}{p} \int_{\Omega} u^{r-p-1} |\nabla u|^{2} \Phi_1 \,\ud x 
- \frac{r}{p} \int_{\Omega} u^{r-p}\nabla u \cdot \nabla \Phi_1 \,\ud x \\
= -\frac{4r(r-p)}{p(r-p+1)^2} \int_{\Omega} |\nabla u^{\frac{r-p+1}{2}}|^{2} \Phi_1 \,\ud x 
- \frac{r\lambda_1}{p(r-p+1)} \int_{\Omega} u^{r-p+1} \Phi_1 \,\ud x 
\leq 0,
\end{multline*}
and so \eqref{eq:monotonicity_norm} holds.
In particular, if $r=p$, it follows from the H\"older inequality that
\begin{align*}
\frac{\ud }{\ud \tau} \int_{\Omega} u^p\,\Phi_1 \,\ud x 
= & \int_{\Omega} \Delta u\,\Phi_1 \,\ud x
=  - \lambda_1 \int_{\Omega} u\,\Phi_1 \,\ud x \\
\geq & -\lambda_1 \Big(\int_{\Omega} u^p\,\Phi_1 \,\ud x \Big)^{1/p} \Big(\int_{\Omega} \Phi_1 \,\ud x \Big)^{(p-1)/p},
\end{align*}
and hence
\begin{equation*}
\frac{\ud }{\ud \tau} \Big(\int_{\Omega} u^p\,\Phi_1 \,\ud x \Big)^{(p-1)/p} \geq -\frac{\lambda_1(p-1)}{p} \Big(\int_{\Omega} \Phi_1 \,\ud x \Big)^{(p-1)/p}.
\end{equation*}
Integrating the above inequality over $\tau \in (t_0, t)$, we complete the proof.
\end{proof}
Consequently, we can prove
\begin{cor}\label{cor:extinction-time}
Let $p \in \big(1, \frac{n+1}{(n-1)_+}\big)$.
Suppose that $u$ is a limit solution to \eqref{eq:Cauchy-Dirichlet} with $ u_0 \in L^{p}_{\Phi_1}(\Omega)$ and $T^*$ denotes the extinction time.
Then for $0 \leq t_0 \leq t < T^*$, 
\begin{equation}\label{eq:monotonicity_norm_p_II}
\begin{split}
\frac{\lambda_1(p-1)}{p}\|\Phi_1\|_{L^1(\Omega)}^{1-p^{-1}} (T^* - t)
\geq & \,\|u(\cdot, t)\|_{L^{p}_{\Phi_1}(\Omega)}^{p-1} \\
\geq & \, \|u(\cdot, t_0)\|_{L^{p}_{\Phi_1}(\Omega)}^{p-1} -\frac{\lambda_1(p-1)}{p}\|\Phi_1\|_{L^1(\Omega)}^{1-p^{-1}} (t - t_0).
\end{split}
\end{equation}
Consequently,
\begin{equation*}
T^* \geq \frac{p}{\lambda_1(p-1)} \|\Phi_1\|_{L^1(\Omega)}^{(1-p)/p} \|u_0\|_{L^{p}_{\Phi_1}(\Omega)}^{p-1}.
\end{equation*}
\end{cor}
By Theorem \ref{thm:weighted-smoothing}, Lemma
\ref{lem:monotonicity}, and Corollary \ref{cor:extinction-time}, we
obtain the following corollary and hence Theorem
\ref{thm:weighted-smoothing_1}.
\begin{cor}\label{cor:weighted-smoothing}
Suppose that the conditions in Theorem \ref{thm:weighted-smoothing} hold.
Then for $0 \leq t_0 < t < \infty$, we have
\begin{equation}\label{equ;u-L-infty-estimate1}
\|u(\cdot,t)\|_{L^{\infty}(\Omega)} 
\leq C_5 ( t - t_0 )^{-\frac{q}{r - q(p-1)}} \Big(\int_{\Omega} u(x,t_0)^{r} \Phi_1 \,\ud x \Big)^{\frac{1}{r- q(p-1)}}.
\end{equation}
Moreover, if $p \in \big(1, \frac{n+1}{(n-1)_+}\big)$ and $r=p$, then for $\varepsilon T^* \leq t  \leq T^*$ with $\varepsilon \in (0,1)$, we have
\begin{equation}\label{equ;u-L-infty-estimate}
\|u(\cdot, t)\|_{L^{\infty}(\Omega)} 
\leq C_6(T^*-t)^{1/(p-1)}.
\end{equation}
Here, $C_5$ is the same constant as in Theorem \ref{thm:weighted-smoothing} and $C_6=C_6(n, \Omega, p, \varepsilon)>0$.
\end{cor}
\begin{proof}
The estimate \eqref{equ;u-L-infty-estimate1} follows from Theorem \ref{thm:weighted-smoothing} and Lemma \ref{lem:monotonicity}.
We divide the proof of \eqref{equ;u-L-infty-estimate} into two cases. If $\varepsilon T^*\leq t\leq T^*/2$, apply \eqref{equ;u-L-infty-estimate1} with $r=p$ and $t_0=t/2$. Using \eqref{eq:monotonicity_norm_p_II}, we obtain
\begin{align*}
\|u(\cdot,t)\|_{L^\infty(\Omega)}
&\leq C(t/2)^{-q/(p-q(p-1))}
\Big(\int_\Omega u(x,t/2)^p\Phi_1\,\ud x\Big)^{1/(p-q(p-1))}\\
&\leq C t^{-q/(p-q(p-1))}(T^*-t/2)^{\frac{p}{(p-1)(p-q(p-1))}}
\leq C(\varepsilon)(T^*-t)^{1/(p-1)}.
\end{align*}
If $T^*/2\leq t<T^*$, take $t_0=2t-T^*$. Then $t-t_0=T^*-t$ and $T^*-t_0=2(T^*-t)$. By \eqref{equ;u-L-infty-estimate1} and \eqref{eq:monotonicity_norm_p_II}, we have
\begin{align*}
\|u(\cdot,t)\|_{L^\infty(\Omega)}
&\leq C(T^*-t)^{-q/(p-q(p-1))}
\Big(\int_\Omega u(x,t_0)^p\Phi_1\,\ud x\Big)^{1/(p-q(p-1))}\\
&\leq C(T^*-t)^{-q/(p-q(p-1))}
(T^*-t_0)^{\frac{p}{(p-1)(p-q(p-1))}}
\leq C(T^*-t)^{1/(p-1)}.
\end{align*}
This completes the proof.
\end{proof}

{
\setlength{\emergencystretch}{2em}
\subsection{A counterexample}\label{subsec:example}
{
In this subsection, we construct boundary-singular separable solutions
that demonstrate the sharpness of the weighted smoothing range.  In the
natural case $r=p$, it follows from the construction that the
Brezis--Turner exponent
\[
p_{\mathrm{BT}}:=\frac{n+1}{n-1}
\]
is the critical exponent: instantaneous $L^\infty$ smoothing fails at
$p=p_{\mathrm{BT}}$.  For $p$ in a right neighborhood of
$p_{\mathrm{BT}}$, it follows from the same construction that the condition
\[
r>r_c:=\frac{(n+1)(p-1)}2
\]
in Theorem \ref{thm:weighted-smoothing_1} cannot be weakened by
decreasing the lower bound $r_c$.

Let $n \geq 2$, $p>1$, and let $\Omega$ be a smooth bounded domain in $\mathbb{R}^n$ with $0 \in \partial \Omega$.
We begin with a positive boundary-singular solution
$S\in C^2(\Omega)\cap C(\overline{\Omega}\setminus\{0\})$ of
\begin{equation}\label{eq:elliptic_equation_singular}
\begin{cases}
-\Delta S = S^p \quad &\text{in}~\Omega,\\
S>0 \quad &\text{in}~\Omega,\\
S=0 \quad &\text{on}~\partial \Omega \setminus \{0\}.
\end{cases}
\end{equation}
By Theorem 1.1 and Propositions 1.1--1.2 of
del Pino--Musso--Pacard \cite{DMP07}, after
decreasing $p_n$ if necessary, there exists
\[
p_n\in\left(p_{\mathrm{BT}},\frac{n+2}{(n-2)_+}\right)
\]
such that \eqref{eq:elliptic_equation_singular} admits such a solution
for every $p\in[p_{\mathrm{BT}},p_n)$.  By the asymptotic description for
$p>p_{\mathrm{BT}}$ in \cite{DMP07}, the endpoint asymptotics in
Theorems 1.2--1.3 of Bidaut-V\'eron--Ponce--V\'eron \cite{BPV},
the boundary Schauder estimates, and the Hopf lemma, there exists
$0<\rho_0<e^{-1}$ such that
\begin{equation}\label{eq:relative-singular-profile}
S(x)\asymp
\begin{cases}
d(x)|x|^{-1-\frac2{p-1}},
&p_{\mathrm{BT}}<p<p_n,\\[1mm]
d(x)|x|^{-n}\left(\log\frac1{|x|}\right)^{-\frac{n-1}{2}},
&p=p_{\mathrm{BT}},
\end{cases}
\end{equation}
for $x\in\Omega$ with $0<|x|<\rho_0$.
}
Let $T^*>0$ and set
$$
{U(x,t)=\Big(\frac{p-1}{p}\Big)^{\frac{1}{p-1}}
\big(T^*-t\big)_+^{\frac{1}{p-1}} S(x).}
$$
{
Then $U$ solves \eqref{eq:Cauchy-Dirichlet} away from the singular
boundary point and extinguishes at time $T^*$.  We consider the
following two parameter regimes:
\[
p_{\mathrm{BT}}<p<p_n,\qquad
p\leq r<r_c:=\frac{(n+1)(p-1)}2,
\]
and the endpoint case $p=r=p_{\mathrm{BT}}=r_c$.  In the first regime,
by \eqref{eq:relative-singular-profile}, we have
$S\in L^r_{\Phi_1}(\Omega)$ for $r<r_c$.  At the endpoint,
\[
\int_0^{\rho_0}
\frac{\ud\rho}
{\rho\bigl(\log(1/\rho)\bigr)^{(n+1)/2}}<\infty,
\]
so $S\in L^{p_{\mathrm{BT}}}_{\Phi_1}(\Omega)$ also in the endpoint
case.  Consequently, in either regime,
$U(\cdot,t)\in L^r_{\Phi_1}(\Omega)$ for every $t\in[0,T^*]$.
}

{It remains to verify that this boundary-singular solution
is the limit solution generated by its initial value.  Since $r\geq p$
in both regimes, $U_0:=U(\cdot,0)$ belongs to
$L^p_{\Phi_1}(\Omega)$.  By Theorem \ref{thm:limit_sol}, there exists
a unique limit solution
$V\in C([0,\infty);L^p_{\Phi_1}(\Omega))$ to
\eqref{eq:Cauchy-Dirichlet} with initial value $U_0$.}
{Choose $\{u_{0,\ell}\}\subset C_c^\infty(\Omega)$ such that
$0\leq u_{0,\ell}\leq\min\{\ell,U_0\}$ and
$u_{0,\ell}\uparrow U_0$ a.e. in $\Omega$.  By monotone convergence,
$u_{0,\ell}^p\to U_0^p$ in $L^1_{\Phi_1}(\Omega)$.  Let $u_\ell$
denote the bounded-data solution with initial value $u_{0,\ell}$.
By monotone approximation and the uniqueness in Theorem
\ref{thm:limit_sol},}
\begin{equation*}
{V(x,t) = \lim_{\ell \to \infty} u_{\ell}(x,t)}
\quad\text{for a.e. }(x,t)\in\Omega\times(0,\infty).
\end{equation*}
{In addition, by the weighted contraction principle in
Theorem \ref{thm:limit_sol}, for every $t\geq0$ we have
\begin{equation*}
0\leq\int_\Omega\bigl(V^p(x,t)-u_\ell^p(x,t)\bigr)\Phi_1(x)\,\ud x
\leq\int_\Omega\bigl(U_0^p-u_{0,\ell}^p\bigr)\Phi_1\,\ud x
\longrightarrow0.
\end{equation*}
}

{We next prove that $V=U$ on
$\Omega\times(0,T^*)$.  Fix $\tau\in(0,T^*)$.  By the regularity
theory in \cite{JX23, JX25},
we have $u_{\ell}(\cdot,t)\in C^2(\overline\Omega)$ for every
$t\in(0,\tau]$.  Fix $\ell\in\mathbb N$, and let
$S_1\in C^2(\overline\Omega)$ be a regular positive solution of
\eqref{eq:elliptic-equation}.  Such a solution exists by the standard
variational theory because $p<p_n<(n+2)/(n-2)_+$.  Since
$S_1\asymp d(\cdot,\partial\Omega)$ near the boundary, there exists
$m\geq1$, depending on $\ell$, such that
$u_{0,\ell}\leq U_m(\cdot,0)$, where}
$$U_m(x,t) = m\Big(\frac{p-1}{p}\Big)^{\frac{1}{p-1}}(T^*-m^{1-p}t)^{\frac{1}{p-1}} S_1(x).$$
{By the comparison principle,
$u_\ell\leq U_m$ on $\Omega\times[0,\tau]$.  Moreover, since
$S_1\asymp d(\cdot,\partial\Omega)$, there exists a constant
$C_{\ell,\tau}>0$ such that
\begin{equation*}
u_{\ell}(x,t) \leq U_m(x,t) \leq C_{\ell,\tau}
\operatorname{dist}(x, \partial \Omega)
\quad\text{for all }(x,t)\in\Omega\times[0,\tau].
\end{equation*}
}
{On the other hand, it follows from the lower bound in
\eqref{eq:relative-singular-profile} and
$\tau<T^*$ that $U(x,t)/d(x,\partial\Omega)\to\infty$ as
$x\to0$, uniformly for $t\in[0,\tau]$.  Hence there exists
$\rho_{\ell,\tau}>0$ such that}
\begin{equation*}
u_{\ell}(x,t) \leq U(x,t)
\quad\text{for all }(x,t)\in
(\Omega\cap B_{\rho_{\ell,\tau}})\times[0,\tau].
\end{equation*}
{Fix $\rho\in(0,\rho_{\ell,\tau})$. On the punctured domain
$\Omega\setminus\overline{B_\rho}$, we have
$u_{0,\ell}\leq U(\cdot,0)$; both functions vanish on
$\partial\Omega\setminus B_\rho$; and, by the preceding estimate,
$u_\ell\leq U$ on the artificial boundary
$\Omega\cap\partial B_\rho$ for all $t\in[0,\tau]$. By the comparison
principle, $u_\ell\leq U$ on
$(\Omega\setminus\overline{B_\rho})\times[0,\tau]$.  Combining this
with the inequality already established in
$(\Omega\cap B_\rho)\times[0,\tau]$, we obtain
\[
u_{\ell}(x,t) \leq U(x,t),
\qquad \forall (x,t)\in\Omega\times[0,\tau].
\]
}
{To pass from $u_\ell\leq U$ to the limit, set}
$$w_{\ell}(x,t) = U(x,t) - u_{\ell}(x,t),\quad F_{\ell}(x,t) = U^p(x,t) - u_{\ell}^p(x,t).$$
{Then $w_\ell,F_\ell\geq0$ and}
\begin{equation}\label{eq:difference_ell}
{
\begin{cases}
\partial_t F_{\ell} = \Delta w_{\ell} & \text{in } \Omega  \times (0,\tau),\\
w_{\ell} = 0 & \text{on }(\partial\Omega\setminus\{0\})\times (0,\tau).
\end{cases}
}
\end{equation}
{Fix $\rho\in(0,\rho_{\ell,\tau})$ and choose
$\chi_\rho\in C^\infty(\mathbb R^n)$ such that}
\begin{equation*}
\chi_{\rho} = 0 \text{ in } \overline{B_{\rho/2}(0)},~~ \chi_{\rho} = 1 \text{ in } \mathbb{R}^n \setminus B_{\rho}(0),\quad |\nabla \chi_{\rho}| \leq C/\rho,~|\Delta \chi_{\rho}| \leq C/\rho^2.  
\end{equation*}
{Test \eqref{eq:difference_ell} with
$\psi:=\chi_\rho\Phi_1$ on
$\Omega_\rho:=\Omega\setminus\overline{B_{\rho/2}}$.  Integrating by
parts, we obtain}
\begin{equation}\label{eq:energy_dif_ieq}
\begin{split}
\frac{\ud }{\ud t} \int_{\Omega_{\rho}} F_{\ell}\psi \,\ud x
& =   \int_{\Omega_{\rho}} \Delta w_{\ell}\psi \,\ud x \\
& =  \int_{\Omega_{\rho}} w_{\ell}\Delta \psi \,\ud x - \int_{\partial\Omega_{\rho}} w_{\ell}\partial_{\nu}\psi \cdot \ud S + \int_{\partial\Omega_{\rho}} \psi \partial_{\nu} w_{\ell} \cdot \ud S \\
& =: \, I_1 + I_2 + I_3.
\end{split}
\end{equation}
{Since $\psi=0$ on $\partial\Omega_\rho$, we have $I_3=0$.
Moreover, $w_\ell=0$ on $\partial\Omega\setminus\{0\}$.  To estimate
the remaining boundary contribution, define}
{
\[
\varepsilon_p(\rho):=
\begin{cases}
\rho^{\,n-1-\frac2{p-1}},&p_{\mathrm{BT}}<p<p_n,\\[1mm]
\bigl(\log(1/\rho)\bigr)^{-\frac{n-1}{2}},&p=p_{\mathrm{BT}}.
\end{cases}
\]
In either case $\varepsilon_p(\rho)\to0$ as $\rho\downarrow0$.
By the upper bound in \eqref{eq:relative-singular-profile}, we have}
\begin{equation*}
{
\begin{aligned}
|I_2| = &\, \Big| \int_{\Omega \cap \partial B_{\rho/2}} w_{\ell}\partial_{\nu}\psi \cdot \ud S \Big|
= \Big| \int_{\Omega \cap \partial B_{\rho/2}} w_{\ell}\Phi_1\partial_{\nu}\chi_{\rho} \cdot \ud S \Big| \\
\leq &\,C\varepsilon_p(\rho)\longrightarrow0
\qquad\text{as }\rho\downarrow0,
\end{aligned}
}
\end{equation*}
{where the estimate is uniform for time in $[0,\tau]$.}
{By the same estimates on the annulus
$B_\rho\setminus B_{\rho/2}$, the two
terms containing $\nabla\chi_\rho$ or $\Delta\chi_\rho$ are
$O(\varepsilon_p(\rho))$, uniformly for time in $[0,\tau]$; the
dominated convergence theorem applies to the remaining term.  Hence}
\begin{equation*}
{
\begin{aligned}
I_1 = \int_{\Omega_{\rho}} w_{\ell}\Delta \psi \,\ud x
= & \int_{\Omega_{\rho}} w_{\ell}\chi_{\rho} \Delta \Phi_1 \,\ud x + \int_{\Omega_{\rho}} w_{\ell}\Phi_1 \Delta \chi_{\rho} \,\ud x + 2\int_{\Omega_{\rho}} w_{\ell}\nabla \chi_{\rho} \cdot \nabla \Phi_1 \,\ud x  \\
= & -\lambda_1 \int_{\Omega_{\rho}} w_{\ell}\chi_{\rho} \Phi_1 \,\ud x + O(\varepsilon_p(\rho))
\rightarrow  -\lambda_1 \int_{\Omega} w_{\ell} \Phi_1 \,\ud x
\end{aligned}
}
\end{equation*}
as $\rho \rightarrow 0$.
{We first integrate \eqref{eq:energy_dif_ieq} over
$(0,t)$ and then let $\rho\downarrow0$.  By the uniform estimates above
and dominated convergence, for every $t\in(0,\tau]$ we have}
\begin{equation*}
\int_{\Omega} (U^p(x,t)-u_{\ell}^p(x,t))\Phi_1 \,\ud x + \lambda_1 \int_{0}^{t}\int_{\Omega} w_{\ell} \Phi_1 \,\ud x = \int_{\Omega} (U^p(x,0)-u_{0,\ell}^p(x))\Phi_1 \,\ud x.
\end{equation*}
{Since $w_\ell\geq0$ and
$u_{0,\ell}^p\to U_0^p$ in $L^1_{\Phi_1}(\Omega)$, we may drop the
nonnegative second term and let $\ell\to\infty$.  Using
$u_\ell\uparrow V$, we obtain, for every $t\in(0,\tau]$,}
\begin{equation*}
{\int_{\Omega} (U^p(x,t)-V^p(x,t))\Phi_1 \,\ud x
\leq \lim_{\ell\to\infty}\int_{\Omega}
(U^p(x,0)-u_{0,\ell}^p(x))\Phi_1 \,\ud x=0,}
\end{equation*}
{
where the integrand on the left is nonnegative because
$u_\ell\leq U$ and hence $V\leq U$.  Therefore
$V(\cdot,t)=U(\cdot,t)$ a.e. in $\Omega$ for every $t\in(0,\tau]$.
Since $\tau<T^*$ was arbitrary, $V=U$ throughout
$\Omega\times(0,T^*)$.  Furthermore,
$V(\cdot,T^*)=0$ by continuity in $L^p_{\Phi_1}(\Omega)$, and
by uniqueness with zero data at time $T^*$, we have $V\equiv0$ for
$t\geq T^*$.  Thus the separable solution $U$ is precisely the limit
solution generated by $U_0$.  Since $S$ is unbounded,
$U(\cdot,t)$ is unbounded for every $0<t<T^*$.

Consequently, for every $p\in(p_{\mathrm{BT}},p_n)$ and every
$r\in[p,r_c)$, there is an initial datum in
$L^r_{\Phi_1}(\Omega)$ whose limit solution remains unbounded before
extinction.  At $p=p_{\mathrm{BT}}$, the same conclusion follows from
the logarithmic profile for $r=p=r_c$.  Hence, in the natural scale $r=p$,
$p<p_{\mathrm{BT}}$ is the sharp range for the
$L^p_{\Phi_1}$--$L^\infty$ smoothing effect: smoothing already fails
at the critical exponent $p=p_{\mathrm{BT}}$.  For
$p\in(p_{\mathrm{BT}},p_n)$, it follows from the counterexamples with
$p\leq r<r_c$ that the lower bound $r_c$ in
Theorem \ref{thm:weighted-smoothing_1} cannot be decreased.  This
construction does not address the borderline case $r=r_c$ for
$p>p_{\mathrm{BT}}$.

\begin{rem}
The restriction \(p<p_n\) above comes only from the currently available
elliptic existence result of del Pino--Musso--Pacard \cite{DMP07}.
As observed in Open Problem~1 of \cite{DMP07}, an affirmative answer
would extend their bounded-domain construction to every
\[
\frac{n+1}{n-1}<p<\frac{n+2}{(n-2)_+}.
\]
With these elliptic profiles, the preceding argument applies without
change, and the same counterexamples are obtained for every
\(p\leq r<\frac{(n+1)(p-1)}2\) throughout this range.
\end{rem}
}
}

\section{Global Harnack Inequalities}\label{sec:Harnack}

In this section, we revisit the global Harnack inequalities for limit solutions to \eqref{eq:Cauchy-Dirichlet} and derive sharper versions.

\begin{lem}
\label{lem:almost-representation}
Suppose that $u$ is a bounded weak solution to \eqref{eq:Cauchy-Dirichlet}.
Then, for a.e. $x_0 \in \Omega$ and $0 < t_0 < t_1 < \infty$,
\begin{equation}
\frac{u(x_0, t_1)}{t_1^{1/(p-1)}} 
\leq \frac{p}{p-1}\int_{\Omega} \frac{u(x,t_1)^p -  u(x,t_0)^p}{t_0^{p/(p-1)} - t_1^{p/(p-1)}} G_{\Omega}(x,x_0) \,\ud x
\leq \frac{u(x_0, t_0)}{t_0^{1/(p-1)}}, 
\end{equation}
where $G_{\Omega}$ denotes the Green function of $-\Delta$ with zero Dirichlet boundary conditions in $\Omega$.
\end{lem}

This formula was established for bounded gradient flow solutions to nonlocal porous medium equations in Bonforte-V\'azquez \cite{BV15} and to nonlocal fast diffusion equations in Bonforte-Ibarrondo-Ispizua \cite{BII23}.
For completeness, we give a proof in the present setting.

\begin{proof}[Proof of Lemma \ref{lem:almost-representation}]
Let $\psi$ be the standard mollifier given by
\begin{equation*}
\psi(x) =
\begin{cases}
C \exp\big( \frac{1}{|x|^2-1} \big) & \text{if}~~|x| < 1, \\
0 & \text{if}~~|x| \geq 1,
\end{cases} 
\end{equation*}
where $C>0$ is a normalization constant such that $\Vert \psi \Vert_{L^1(\mathbb{R}^n)}= 1$.
Define a sequence of mollifiers $\{\psi_k\}$ by
$$\psi_k(x) = k^n \psi(kx)$$
and so $\psi_k \in C_c^{\infty}(\mathbb{R}^n)$ with $ \operatorname{supp} \psi_k \subset B_{{1}/{k}}(0).$ 
Recall the following property of the mollifiers $\{\psi_k\}$:
if $f \in L^1_{\mathrm{loc}}(\Omega)$, then
\[
\int_{\Omega} f(y)\,\psi_k(y-x) \,\ud y
\longrightarrow f(x)\quad\text{a.e. in }\Omega
\quad\text{as }k\to\infty.
\]
For $k \in \mathbb{N}$ and $x, x_0 \in \Omega$, set
$$\phi_k(x, x_0):= \int_{\Omega} G_{\Omega}(x,y)\,\psi_k(y - x_0) \,\ud y,$$
which is well-defined as $G_{\Omega}(x,\cdot) \in L^1(\Omega)$ for all $x \in \Omega$.
Indeed, $\phi_k$ is a solution to the following equation:
\begin{equation}\label{eq:possion}
\begin{cases}
-\Delta \phi_k(x) = \psi_k(x-x_0) \quad &\text{in~} \Omega,\\
\phi_k(x) = 0 \quad &\text{on~} \partial\Omega.
\end{cases}
\end{equation}
By the standard elliptic regularity theory, we have $\phi_k \in C_0(\Omega) \cap C^2(\overline{\Omega})$.
Choosing $\varphi=\phi_k$ in the weak formulation
\eqref{eq:weak-sol}, we obtain, for $0 \leq t_0 < t_1 < \infty$,
\begin{align*}
&\int_{\Omega} u(x,t_1)^p\,\phi_k(x, x_0) \,\ud x
- \int_{\Omega} u(x,t_0)^p\,\phi_k(x, x_0) \,\ud x \\
= & -\int_{t_0}^{t_1}\!\int_{\Omega} \nabla u(x,\tau) \cdot
\nabla_x \phi_k(x,x_0) \,\ud x \, \ud \tau \\
= & -\int_{t_0}^{t_1}\!\int_{\Omega}  u(x,\tau) \psi_k(x-x_0) \,\ud x \, \ud \tau.
\end{align*}
where we have used the weak formula of \eqref{eq:possion}.
By the approximation property of the mollifiers, we have
\begin{align*}
\int_{t_0}^{t_1}\!\int_{\Omega} u(x,\tau)\,\psi_k(x - x_0) \,\ud x \, \ud \tau 
\rightarrow \int_{t_0}^{t_1} u(x_0, \tau) \,\ud \tau, \quad \text{as~~}k \to \infty.
\end{align*}
Note that
\begin{align*}
\int_{\Omega}\Big(\int_{\Omega} u(x,t)^p G_{\Omega}(x,y) \,\ud x \Big) \,\ud y = & \int_{\Omega} u(x,t)^p \Big( \int_{\Omega} G_{\Omega}(x,y) \,\ud y \Big) \,\ud x \\
\leq & C \int_{\Omega} u(x,t)^p \Phi_1(x) \,\ud x < \infty,
\end{align*}
where we have used the Fubini-Tonelli theorem and the estimate $\|G_{\Omega}(x, \cdot)\|_{L^1(\Omega)} \le C\Phi_1(x)$ from Bonforte--Figalli--V\'azquez \cite{BFV18}.
By the Fubini--Tonelli theorem and the properties of mollifiers, for
any $t\geq0$ we have
\begin{align*}
\int_{\Omega} u(x,t)^p\,\phi_k(x, x_0) \,\ud x 
& =  \int_{\Omega} u(x,t)^p\int_{\Omega} G_{\Omega}(x,y)\,\psi_k(y - x_0) \,\ud y \,\ud x \\
&=  \int_{\Omega} \psi_k(y - x_0) \int_{\Omega} u(x,t)^p\,G_{\Omega}(x,y) \,\ud x \,\ud y \\
&\rightarrow  \int_{\Omega} u(x,t)^p\,G_{\Omega}(x,x_0) \,\ud x \quad \text{as~~}k \to \infty.
\end{align*}
Consequently, for a.e. $x_0 \in \Omega$ and
$0<t_0<t_1<\infty$, we have
\begin{equation*}
\int_{\Omega} \left(u(x,t_1)^p -  u(x,t_0)^p\right)G_{\Omega}(x,x_0) \,\ud x
= -\int_{t_0}^{t_1} u(x_0,\tau) \,\ud \tau.
\end{equation*}
Recall the B\'enilan-Crandall estimate \cite{BC81} (see also Section 3 in
Bonforte--Figalli \cite{BF24}) for bounded weak solutions to
\eqref{eq:Cauchy-Dirichlet}, i.e.,
$$ u_t \leq \frac{u}{(p-1)t},$$
from which it follows that $t \mapsto t^{-1/(p-1)} u(x,t)$ is nonincreasing for a.e. $x_0 \in \Omega$.
As a result, for $0< t_0 < \tau <t_1$ and a.e. $x_0 \in \Omega$,
\begin{align*}
\Big(\frac{\tau}{t_1}\Big)^{\frac{1}{p-1}}u(x_0,t_1) \leq u(x_0, \tau) \leq \Big(\frac{\tau}{t_0}\Big)^{\frac{1}{p-1}}u(x_0,t_0),
\end{align*}
and hence
\begin{align*}
\frac{p-1}{p}\big(t_1^{\frac{p}{p-1}}-t_0^{\frac{p}{p-1}}\big)\frac{u(x_0,t_1)}{t_1^{1/(p-1)}}= &\,\frac{u(x_0,t_1)}{t_1^{1/(p-1)}}\,\int_{t_0}^{t_1}\tau^{\frac{1}{p-1}}\,\ud\, \tau \leq \int_{t_0}^{t_1} u(x_0,\tau) \,\ud \tau, \\
\frac{p-1}{p}\big(t_1^{\frac{p}{p-1}}-t_0^{\frac{p}{p-1}}\big)\frac{u(x_0,t_0)}{t_0^{1/(p-1)}}= &\,\frac{u(x_0,t_0)}{t_0^{1/(p-1)}}\,\int_{t_0}^{t_1}\tau^{\frac{1}{p-1}}\,\ud\, \tau \geq \int_{t_0}^{t_1} u(x_0,\tau) \,\ud \tau.
\end{align*}
Therefore, for $0<t_0<t_1$ and a.e. $x_0\in\Omega$,
\begin{equation*}
\frac{u(x_0, t_1)}{t_1^{1/(p-1)}} 
\leq \frac{p}{p-1}\int_{\Omega} \frac{u(x,t_1)^p -  u(x,t_0)^p}{t_0^{p/(p-1)} - t_1^{p/(p-1)}} G_{\Omega}(x,x_0) \,\ud x
\leq \frac{u(x_0, t_0)}{t_0^{1/(p-1)}}.
\end{equation*}
\end{proof}

We first use Lemma \ref{lem:almost-representation} to show that
the lower boundary estimates of bounded weak solutions to
\eqref{eq:Cauchy-Dirichlet} are controlled by their lower bounds of
the $L^p_{\Phi_1}$ norms.
\begin{lem} \label{lm:global-lower-bounds-I}
Let $u$ be a bounded weak solution to \eqref{eq:Cauchy-Dirichlet}, and let $T^*$ be the extinction time.
Then for $0 < t < T^*$ and a.e. $x \in \Omega$,
\begin{equation}\label{equ;u/Phi1-intergral}
\frac{u(x,t)}{\Phi_1(x)}
\geq 
\frac{C \, t^{\frac{1}{p-1}}}{T^{*\frac{p}{p-1}} - t^{\frac{p}{p-1}}}
\int_{\Omega} u^p(y,t)\,\Phi_1(y)\,\ud y.
\end{equation}
where $C > 0$ is a constant depending only on $n$, $p$ and $\Omega$.
\end{lem}

\begin{proof}
Apply Lemma \ref{lem:almost-representation} with $t_0=t$ and then let $t_1\uparrow T^*$. Since $u(\cdot,t_1)\to0$, for any $t \in (0,T^*)$ and a.e. $x\in\Omega$ we obtain
\[
\frac{u(x, t)}{t^{\frac{1}{p-1}}} \geq \frac{p}{p-1} \int_{\Omega} \frac{u^p(y, t)}{T^{*\frac{p}{p-1}} - t^{\frac{p}{p-1}}}\, G_{\Omega}(x, y) \,\ud y .
\]
Using the lower bound for the Green function due to Davies
\cite{D87} and Zhang \cite{Zh02},
$G_{\Omega}(x, y) \geq c_0 \Phi_1(x) \Phi_1(y)$, we obtain
\[
\frac{u(x, t)}{t^{\frac{1}{p-1}}} \geq \Phi_1(x)\,\frac{c_0 p}{p-1}\, \frac{1}{T^{*\frac{p}{p-1}} - t^{\frac{p}{p-1}}} \int_{\Omega} u^p(y, t)\,\Phi_1(y) \,\ud y .
\]
Rearranging the terms, we obtain \eqref{equ;u/Phi1-intergral}; the corresponding constant is $C=c_0p/(p-1)$.
\end{proof}

We now use Lemma \ref{lem:almost-representation} and Lemma \ref{lm:global-lower-bounds-I} to prove the main results in this section.
\begin{thm}\label{thm:bhi_gengral}
Suppose that $u$ is a limit solution to \eqref{eq:Cauchy-Dirichlet} with $u_0 \in L^{r}_{\Phi_1}(\Omega)$ and $r>0$ satisfying \eqref{eq:initial_datum_range_2}.
Let $T^*$ denote the extinction time. 
Then, for $0 < t < T^*$ and a.e. $x \in \Omega$,
\begin{equation}\label{eq:bdi_general}
     C_1\frac{ t^{\frac{1}{p-1}} T^{*a+b+1}}{T^{*\frac{p}{p-1}} - t^{\frac{p}{p-1}}} {B_{1-\frac{t}{T^*}}(a+1,b+1)}  \|u_0\|_{L^{r}_{\Phi_1}(\Omega)}^{-\mu}
    \leq  \frac{u(x, t)}{\Phi_1(x)} 
    \leq  C_2 \, t^{-c} \|u_0\|_{L^{r}_{\Phi_1}(\Omega)}^{\frac{pr}{r - q(p-1)}},
\end{equation}
where $C_1$ and $C_2$ are positive constants depending only on $n$, $p$, $r$, $\Omega$, and $q$. 
Here $q=(n+1)/2$ if $n\geq2$, while for $n=1$ we fix any
$q\in\bigl(1,r/(p-1)\bigr)$. Moreover,
$$
a := \frac{r+1}{p-1}, \qquad
b := \frac{r(q+1)}{r-q(p-1)}, \qquad
c := \frac{r+q}{r-q(p-1)}, \qquad
\mu := \frac{r(r+p-1)}{r-q(p-1)},
$$
{and $B_z(\alpha,\beta)$ denotes the incomplete Beta function,
\[
    B_z(\alpha,\beta) := \int_0^{z} s^{\alpha-1}(1-s)^{\beta-1} \, \ud s.
\]}
\end{thm}

\begin{proof}
Set $u_{0,k}:=\min\{u_0,k\}$, let $u_k$ be the corresponding bounded
solutions, and denote their extinction times by $T_k^*$.  By
the comparison principle and the definition of a limit solution,
$u_k\uparrow u$ a.e. and
$T_k^*\leq T_{k+1}^*\leq T^*$.  If
$\overline T:=\lim_kT_k^*<T^*$, then $u_k(\cdot,t)=0$ for every
$k$ and $t>\overline T$, and hence $u(\cdot,t)=0$, a contradiction.
Thus $T_k^*\uparrow T^*$.  Moreover,
$\|u_{0,k}-u_0\|_{L^r_{\Phi_1}(\Omega)}\to0$.  For each fixed
$t\in(0,T^*)$, we have $t<T_k^*$ for all sufficiently large $k$;
therefore it suffices to prove \eqref{eq:bdi_general} for bounded weak
solutions and then let $k\to\infty$.

We first prove the upper bound.  By Lemma
\ref{lem:almost-representation} and the estimate
$\|G_{\Omega}(x,\cdot)\|_{L^1(\Omega)}\leq C\Phi_1(x)$
from Bonforte--Figalli--V\'azquez \cite{BFV18}, for
$0<t_0<t_1<\infty$ we have
\begin{equation}\label{eq:up_bound}
\frac{u(x,t_1)}{\Phi_1(x)} 
\le \frac{C p}{p-1} \frac{t_1^{1/(p-1)}}{t_1^{p/(p-1)}-t_0^{p/(p-1)}} \|u(\cdot,t_0)\|_{L^\infty(\Omega)}^p .  
\end{equation}
Since the function $\tau \mapsto \tau^{p/(p-1)}$ is convex, we have
$t_1^{p/(p-1)}-t_0^{p/(p-1)} \ge (t_1-t_0) t_0^{1/(p-1)}$. Hence,
\[
\frac{u(x,t_1)}{\Phi_1(x)} 
\le \frac{C p}{p-1} \Bigl(\frac{t_1}{t_0}\Bigr)^{\!1/(p-1)} \frac{\|u(\cdot,t_0)\|_{L^\infty(\Omega)}^p}{\,t_1-t_0\,}.
\]
Fix $t \in (0,T^*)$ and choose $t_1=t$, $t_0=t/2$. Then $t_1/t_0=2$ and $t_1-t_0=t/2$.
By Theorem~\ref{thm:weighted-smoothing_1}, we have
\[
\|u(\cdot, t/2)\|_{L^\infty(\Omega)} \le C \, t^{-\frac{q}{\,r-q(p-1)\,}} \|u_0\|_{L^{r}_{\Phi_1}(\Omega)}^{\frac{r}{\,r-q(p-1)\,}} .
\]
Substituting this into the preceding inequality, we obtain
\begin{equation}\label{eq:bdi_general-left}
\frac{u(x,t)}{\Phi_1(x)} 
\le C \, t^{-1-\frac{pq}{r-q(p-1)}} \|u_0\|_{L^{r}_{\Phi_1}(\Omega)}^{\frac{pr}{r-q(p-1)}}
= C \, t^{-\frac{r+q}{r-q(p-1)}} \|u_0\|_{L^{r}_{\Phi_1}(\Omega)}^{\frac{pr}{r-q(p-1)}}, \end{equation}
which is exactly the right-hand side of \eqref{eq:bdi_general} (recall that $c=\frac{r+q}{r-q(p-1)}$).

We next prove the lower bound.
Note that for $t>0$,
\begin{align*}
 \frac{\ud}{\ud t} \int_{\Omega} u(x,t)^p\,\Phi_1(x)\, \ud x 
 = & -\lambda_1 \int_{\Omega} u(x,t)\,\Phi_1(x)\, \ud x \\
 = & -\lambda_1 \int_{\Omega} \frac{u(x,t)^{r+1}}{u(x,t)^{r-1}\, \big[\frac{u(x,t)}{\Phi_1(x)}\big]} \,\ud x \\
 \leq & -\lambda_1 \frac{\|u(\cdot, t)\|_{L^{r+1}(\Omega)}^{r+1}}{\|u(\cdot, t)\|_{L^{\infty}(\Omega)}^{r-1}\,\|\frac{u(\cdot,t)}{\Phi_1} \|_{L^{\infty}(\Omega)}}.
\end{align*}
By Lemma 8.5 in Bonforte--Figalli \cite{BF24}, there exists a constant $C>0$ depending only on $n,p,r$ and $\Omega$ such that
$$\|u(\cdot, t)\|_{L^{r+1}(\Omega)} \geq C (T^*-t)^{\frac{1}{p-1}}.$$ 
By Theorem~\ref{thm:weighted-smoothing_1} and \eqref{eq:bdi_general-left}, we obtain
\begin{equation}\label{eq:two_bounds}
    \|u(\cdot, t)\|_{L ^{\infty}(\Omega)}^{r-1}\,
    \Big\|\frac{u(\cdot,t)}{\Phi_1}\Big\|_{L^{\infty}(\Omega)}
    \leq
    C t^{-b}
    \|u_0\|_{L^{r}_{\Phi_1}(\Omega)}^{\mu} .
\end{equation}
Combining the preceding differential inequality with the last two
estimates, for $0<t<T^*$ we obtain
\begin{equation}\label{equ;d/dt1}
\begin{aligned}
    \frac{\ud}{\ud t}
    \int_\Omega u(x,t)^p \Phi_1(x)\,\ud x
    \leq
    -\, C (T^*-t)^{a}
    t^{b}
     \|u_0\|_{L^{r}_{\Phi_1}(\Omega)}^{-\mu} .
\end{aligned}
\end{equation}
Integrating \eqref{equ;d/dt1} over the time interval $(t,T^*)$, we obtain
\begin{align*}
\int_\Omega u(x,t)^p \Phi_1(x)\,\ud x
&\geq C\|u_0\|_{L^{r}_{\Phi_1}(\Omega)}^{-\mu}
\int_t^{T^*} (T^*-s)^a\,s^b\,\ud s .
\end{align*}
Setting $s=T^*\tau$ and then $\sigma=1-\tau$, we obtain
\begin{align*}
\int_t^{T^*} (T^*-s)^a\,s^b\,\ud s 
=T^{*a+b+1} \int_{t/T^*}^{1}\,(1-\tau)^a\,\tau^b\,\ud\tau  
{= T^{*a+b+1}\,B_{1-\frac{t}{T^*}}(a+1,b+1).}
\end{align*} 
Consequently, for $0<t<T^*$ we have
\begin{align*}
\int_\Omega u(x,t)^p\,\Phi_1(x)\,\ud x
\ge C\, T^{*a+b+1} \|u_0\|_{L^{r}_{\Phi_1}(\Omega)}^{-\mu} 
\,{B_{1-\frac{t}{T^*}}(a+1,b+1)}.
\end{align*}
By Lemma \ref{lm:global-lower-bounds-I} and the preceding inequality,
we obtain
\begin{align*}
\frac{u(x,t)}{\Phi_1(x)} \geq & \,\frac{C \, t^{\frac{1}{p-1}}}{T^{*\frac{p}{p-1}} - t^{\frac{p}{p-1}}}\int_{\Omega} u^p(y,t)\,\Phi_1(y)\,\ud y \\
\geq & \,C_1\frac{ t^{\frac{1}{p-1}} T^{*a+b+1}}{T^{*\frac{p}{p-1}} - t^{\frac{p}{p-1}}} {B_{1-\frac{t}{T^*}}(a+1,b+1)}  \|u_0\|_{L^{r}_{\Phi_1}(\Omega)}^{-\mu}.
\end{align*}
This completes the proof.
\end{proof} 

Finally, we establish a global Harnack inequality near the extinction time $T^*$ for $p \in \big(1, \frac{n+1}{(n-1)_+}\big)$.
A similar result was proved in Theorem 8.20 of Bonforte--Figalli
\cite{BF24} for the more restrictive range
$p \in (1, \frac{n}{n-1})$ with $n \geq 3$.
\begin{thm}\label{thm:bhi_special}
Let $p \in \big(1, \frac{n+1}{(n-1)_+}\big)$. 
Assume that $u$ is a limit solution to \eqref{eq:Cauchy-Dirichlet} with  $u_0 \in L^p_{\Phi_1}(\Omega)$. 
Then there exist constants $C_i>0$, $i=1,2$, depending only on $n$, $p$, and $\Omega$, such that for  $t \in (2T^*/3, T^*)$,
\begin{equation}\label{equ;u/Phi-2}
C_1 \, (T^* - t)^{\frac{1}{p-1}} \leq \frac{u(x, t)}{\Phi_1(x)} \leq C_2 \, (T^* - t)^{\frac{1}{p-1}},\quad \forall x \in \Omega.
\end{equation}
\end{thm}

\begin{proof}
By the same bounded approximation and extinction-time argument used in the proof of Theorem \ref{thm:bhi_gengral}, it suffices to consider bounded weak solutions.
For $t_0\geq T^*/3$, substitute Corollary
\ref{cor:weighted-smoothing} with $\varepsilon=1/3$ into
\eqref{eq:up_bound} to obtain
\begin{equation*}
\frac{u(x,t_1)}{\Phi_1(x)} 
\leq C \Big(\frac{t_1}{t_0}\Big)^{1/(p-1)} \frac{ (T^* - t_0)^{p/(p-1)} }{t_1 - t_0}.
\end{equation*}
Set $t_1 = t$ and $t_0 = 2t - T^*$. Since $t \geq 2T^*/3$, we have $t_0 \geq T^*/3 > 0$. Note that $t_1 - t_0 = T^* - t$ and $T^* - t_0 = 2(T^* - t)$. Hence,
\begin{align*}
\frac{u(x,t)}{\Phi_1(x)} 
&\leq C \Big(\frac{t}{2t-T^*}\Big)^{1/(p-1)} \frac{  [2(T^* - t) ]^{p/(p-1)} }{T^* - t} \\
&= C 2^{p/(p-1)} \Big(\frac{t}{2t-T^*}\Big)^{1/(p-1)} (T^* - t)^{\frac{1}{p-1}}.
\end{align*}
The function $t \mapsto \frac{t}{2t-T^*}$ is nonincreasing on $(T^*/2, T^*)$, and since $t \geq 2T^*/3 > T^*/2$, we obtain
\[
\frac{t}{2t-T^*} \leq \frac{2T^*/3}{2\cdot 2T^*/3 - T^*} = 2.
\]
Consequently,
\[
\frac{u(x,t)}{\Phi_1(x)} \leq C 2^{\frac{p+1}{p-1}} (T^* - t)^{\frac{1}{p-1}}.
\]
Choosing $C_2 = C 2^{\frac{p+1}{p-1}}$ establishes the right-hand side of \eqref{equ;u/Phi-2}.

We now prove the lower bound. Combining \eqref{equ;u-L-infty-estimate} with the upper bound just proved, we obtain
\[
\|u(\cdot, t)\|_{L^{\infty}(\Omega)}^{p-1}
\Big\|\frac{u(\cdot,t)}{\Phi_1}\Big\|_{L^{\infty}(\Omega)}
\le
C (T^*-t) (T^*-t)^{\frac{1}{p-1}}
=
C (T^*-t)^{\frac{p}{p-1}}.
\]
Using this estimate in place of \eqref{eq:two_bounds}, we obtain
\begin{equation}\label{equ;d/dt2}
\frac{\mathrm{d}}{\mathrm{d} t} \int_\Omega u(x,t)^p \Phi_1(x)\,\mathrm{d} x
\le -\, C (T^*-t)^{\frac{1}{p-1}}.
\end{equation}
Integrating \eqref{equ;d/dt2} from $t$ to $T^*$, we obtain
\[
\int_{\Omega} u(x,t)^p\,\Phi_1(x)\,\mathrm{d} x
\ge
C \,(T^*-t)^{\frac{p}{p-1}}.
\]
Observe that
\[
\frac{t^{\frac{1}{p-1}}}{T^{*\frac{p}{p-1}}-t^{\frac{p}{p-1}}} 
\ge \Big[\frac{p-1}{p}\Big(\frac{2}{3}\Big)^{\frac{1}{p-1}}\Big] \frac{1}{T^*-t}.
\]
By Lemma~\ref{lm:global-lower-bounds-I}, we obtain the left-hand side
estimate $C_1 (T^*-t)^{\frac{1}{p-1}}$ for some constant $C_1>0$
depending only on $n$, $p$, and $\Omega$.
This completes the proof.
\end{proof}


\section{Ancient Solutions}\label{sec:Ancient}
In this section, we first establish a priori estimates and boundary bounds
for positive ancient solutions to \eqref{eq:ADP}.  We then analyze the
backward limits of the rescaled flow and prove the polynomial/exponential
convergence-rate dichotomy in Theorem \ref{thm:asymptotic_behaviour}.
Finally, we establish the rigidity criteria in Theorem
\ref{thm:ancient-rigidity}.
\subsection{Basic estimates}\leavevmode
\begin{lem}\label{lm:norm_bounds_basic}
Suppose that $r\in\mathbb{R}$ satisfies $r >p$ and $r \geq \frac{n(p-1)}{2}$.
If $u$ is a positive ancient solution to \eqref{eq:ADP}, then there exist constants
$C_1=C_1(n,p,r,\Omega)>0$ and $C_2=C_2(n,p,\Omega)>0$
such that for all $t \in (-\infty, 0]$,
\begin{equation}\label{eq:r_norm_lower_bound}
\|u(\cdot, t)\|_{L^{r}(\Omega)} \geq C_1 (-t)^{ \frac{1}{p-1} }, 
\end{equation}
and
\begin{equation}\label{eq:goal-weighted-up-pointwise}
\|u(\cdot, t)\|_{L^{p}_{\Phi_1}(\Omega)}
\le C_2 (-t)^{\frac{1}{p-1}}.
\end{equation}
\end{lem}
\begin{proof}
It follows from \eqref{eq:ADP} and integration by parts that there
exists a constant $C>0$, depending on $n$, $p$, $r$, and $\Omega$,
such that
\begin{equation}\label{eq:derivative_r_norm}
\begin{split}
\frac{\ud}{\ud t} \int_\Omega u^{r} \,\mathrm{d}x
= & \,r \int_\Omega u^{r-1} \partial_t u \,\mathrm{d}x
=\frac{r}{p} \int_\Omega u^{r-p} \Delta u \,\mathrm{d}x \\
= & -\frac{r(r-p)}{p} \int_\Omega u^{r-p-1} |\nabla u|^2 \,\mathrm{d}x
= -\frac{4r(r-p)}{p(r-p+1)^2} \int_\Omega | \nabla u^{\frac{r-p+1}{2}} |^2 \,\mathrm{d}x \\
\leq & - C \Big( \int_\Omega u^{r} \,\mathrm{d}x \Big)^{\frac{r-p+1}{r}},
\end{split}
\end{equation}
where the last inequality holds by the Sobolev inequality and H\"older inequality, under the assumptions $r >p$ and $r \geq \frac{n(p-1)}{2}$.
Equivalently, this reduces to
\begin{equation*}
\frac{\mathrm{d}}{\mathrm{d}t} \|u( \cdot ,t )\|_{L^{r}(\Omega)}^{p-1} \leq - C,
\end{equation*}
from which \eqref{eq:r_norm_lower_bound} follows immediately by integration from $t$ to $0$.
Similarly, we have
\begin{align*}
\frac{\ud}{\ud t} \int_\Omega u(x,t)^p\,\Phi_1(x)\, \mathrm{d}x
= & \int_\Omega \partial_t u^p(x,t)\,\Phi_1(x)\, \mathrm{d}x 
=  \int_\Omega \Delta u(x,t)\,\Phi_1(x)\, \mathrm{d}x \\
= & -\lambda_1 \int_\Omega u(x,t)\,\Phi_1(x)\, \mathrm{d}x \\
\geq & -\lambda_1 \Big( \int_\Omega u(x,t)^p\,\Phi_1(x)\, \mathrm{d}x \Big)^{1/p} \Big( \int_\Omega \Phi_1(x)\, \mathrm{d}x \Big)^{(p-1)/p}.
\end{align*}
Equivalently,
\begin{equation*}
\frac{\ud}{\ud t} \Big( \int_\Omega u(x,t)^p\,\Phi_1(x)\, \mathrm{d}x \Big)^{(p-1)/p}
\geq -\frac{\lambda_1(p-1)}{p} \Big( \int_\Omega \Phi_1(x)\, \mathrm{d}x \Big)^{(p-1)/p},
\end{equation*}
and hence for any $t\in (-\infty, 0]$,
\begin{equation*}
\Big( \int_\Omega u(x,t)^p\,\Phi_1(x)\, \mathrm{d}x \Big)^{(p-1)/p}
\leq \frac{\lambda_1(p-1)}{p} \Big( \int_\Omega \Phi_1(x)\, \mathrm{d}x \Big)^{(p-1)/p} (-t),
\end{equation*}
This proves the desired estimate and completes the proof.
\end{proof}

\begin{lem}\label{lm:noninvreasing_norm_p+1}
Let $u$ be a positive ancient solution to \eqref{eq:ADP}.
Then the mapping $$t \mapsto (-t)^{-\frac{1}{p-1}} \|u(\cdot,t)\|_{L^{p+1}(\Omega)}$$ 
is nonincreasing on $(-\infty, 0)$.
\end{lem}
\begin{proof}
Our proof is inspired by the proof of Lemma 2 in Berryman--Holland
\cite{BH80}.
Using \eqref{eq:ADP} and integration by parts, we have
\begin{equation}\label{eq:derivetion_nabla}
\begin{split}
\frac{\ud }{\ud t} \int_\Omega |\nabla u(x,t)|^2 \,\ud x 
= & \,2 \int_\Omega \nabla u(x,t) \cdot \nabla u_t(x,t) \,\ud x
= -2 \int_\Omega u_t(x,t) \Delta u(x,t) \,\ud x \\
= & -\frac{2}{p} \int_\Omega \frac{(u(x,t)\Delta u(x,t))^2}{u(x,t)^{1+p}} \,\ud x \\
\leq & \,-\frac{2}{p} \Big(\int_\Omega |\nabla u(x,t)|^2 \,\ud x\Big)^2 \Big(\int_\Omega u(x,t)^{p+1} \,\ud x \Big)^{-1},
\end{split}
\end{equation}
where we have used the Cauchy-Schwarz inequality in the last inequality.
Applying the same argument with $r=p+1$ in
\eqref{eq:derivative_r_norm}, we obtain
\begin{equation}\label{eq:identity_norm}
\frac{\ud}{\ud t} \int_\Omega u(x,t)^{p+1} \,\ud x
= -\frac{p+1}{p} \int_\Omega |\nabla u(x,t)|^2 \,\ud x.
\end{equation}
Using \eqref{eq:derivetion_nabla} and \eqref{eq:identity_norm}, we arrive at
\begin{equation*}
\frac{\ud }{\ud t} \Big(\ln \int_\Omega |\nabla u(x,t)|^2 \,\ud x\Big) 
\leq  \frac{\ud }{\ud t} \Big(\ln \Big( \int_\Omega u(x,t)^{p+1} \,\ud x\Big)^{\frac{2}{p+1}}\Big),
\end{equation*}
It follows that the mapping $t \mapsto \mathcal{Q}[u(\cdot, t)]$ is nonincreasing on $(-\infty, 0]$, where $\mathcal{Q}$ is the nonlinear Rayleigh quotient defined by
\[
\mathcal{Q}[u] = \frac{\|\nabla u\|_{L^2(\Omega)}^2}{\| u\|_{L^{p+1}(\Omega)}^2}.
\]
Using \eqref{eq:identity_norm} again, we have
\begin{equation*}
\frac{\ud }{\ud t}\Big(\int_{\Omega}u(x,t)^{p+1} \,\ud x\Big)^{\frac{p-1}{p+1}}
= -\frac{p-1}{p}\mathcal{Q}[u(\cdot, t)], 
\end{equation*}
and so 
\begin{equation*}
\frac{\ud^2 }{\ud t^2}\Big(\int_{\Omega}u(x,t)^{p+1} \,\ud x\Big)^{\frac{p-1}{p+1}}
= -\frac{p-1}{p}\frac{\ud }{\ud t}\mathcal{Q}[u(\cdot, t)] \geq 0.
\end{equation*}
Therefore, the mapping $t \mapsto \|u(\cdot,t)\|_{L^{p+1}(\Omega)}^{p-1}$ is convex on $(-\infty, 0]$. 
Hence for $-\infty<t_1 < t_2 < 0$,
\begin{equation*}
\|u(\cdot,t_2)\|_{L^{p+1}(\Omega)}^{p-1} \leq \frac{\|u(\cdot,t_1)\|_{L^{p+1}(\Omega)}^{p-1}}{-t_1}(-t_2),
\end{equation*}
or equivalently,
\begin{equation*}
(-t_2)^{-\frac{1}{p-1}} \|u(\cdot,t_2)\|_{L^{p+1}(\Omega)} \leq (-t_1)^{-\frac{1}{p-1}} \|u(\cdot,t_1)\|_{L^{p+1}(\Omega)}. 
\end{equation*}
This completes the proof.
\end{proof}

\begin{lem}\label{lm:norm_infty_upper_bound}
Let $\varepsilon>0$.
Suppose that $u$ is a positive ancient solution to \eqref{eq:ADP} such that one of the following conditions holds:
\begin{itemize}
\item[(i)] There exists $r$ satisfying \eqref{eq:initial_datum_range_1} such that
\[
c_1:=\limsup_{t \to -\infty}\,(-t)^{-\frac{1}{p-1}} \|u(\cdot,t)\|_{L^{r}(\Omega)} < \infty.
\]
\item[(ii)] There exists $r$ satisfying \eqref{eq:initial_datum_range_2} such that
\[
c_2:=\limsup_{t \to -\infty}\,(-t)^{-\frac{1}{p-1}} \|u(\cdot,t)\|_{L^{r}_{\Phi_1}(\Omega)} < \infty.
\]
\end{itemize}
Then, for every $t \in (-\infty,-\varepsilon]$,
\begin{equation}\label{eq:L_infty_up}
\|u(\cdot,t)\|_{L^{\infty}(\Omega)} \le C_{\infty}(-t)^{\frac{1}{p-1}},
\end{equation}
{Here $C_\infty$ may depend on $\varepsilon$ and $u$.}
{Furthermore, if
$p \in \big(1,\frac{n+2}{(n-2)_+}\big)$, then
\eqref{eq:L_infty_up} holds for all $t<0$ with $C_\infty$
independent of $\varepsilon$.}
\end{lem}

\begin{proof}
{
By the limsup assumptions, there exist $t_0<0$ and constants $b_i>0$
such that, for every $t\leq t_0$,
\begin{align*}
\|u(\cdot,t)\|_{L^{r}(\Omega)} \le b_1(-t)^{\frac{1}{p-1}},
\quad &\text{if \textup{(i)} holds},\\
\|u(\cdot,t)\|_{L^{r}_{\Phi_1}(\Omega)} \le b_2(-t)^{\frac{1}{p-1}},
\quad &\text{if \textup{(ii)} holds}.
\end{align*}
Fix $A>1$ such that $-A\varepsilon\leq t_0$.  Then $At\leq t_0$
whenever $t\leq-\varepsilon$.  Under condition \textup{(i)}, applying
\eqref{eq:global_smoothing_effects} to the time interval $[At,t]$
gives
\begin{align*}
\|u(\cdot,t)\|_{L^\infty(\Omega)}
&\leq C(t-At)^{-\frac{n}{2r-n(p-1)}}
\|u(\cdot,At)\|_{L^r(\Omega)}^{\frac{2r}{2r-n(p-1)}}\\
&\leq C(-t)^{\frac1{p-1}}.
\end{align*}
Under condition \textup{(ii)}, Theorem
\ref{thm:weighted-smoothing_1}, applied on the same time interval,
yields
\begin{align*}
\|u(\cdot,t)\|_{L^\infty(\Omega)}
&\leq C(t-At)^{-\frac{q}{r-q(p-1)}}
\|u(\cdot,At)\|_{L^r_{\Phi_1}(\Omega)}^{
\frac{r}{r-q(p-1)}}\\
&\leq C(-t)^{\frac1{p-1}}.
\end{align*}
This proves the first assertion.
}

{Now assume
$p\in\big(1,\frac{n+2}{(n-2)_+}\big)$.  Taking $b_i=c_i+1$ in the
corresponding case and applying the preceding argument with $A=2$, we
obtain, for all sufficiently negative $t$,
\[
\|u(\cdot,t)\|_{L^\infty(\Omega)}
\leq M_i(-t)^{\frac1{p-1}}
\]
with $M_i$ independent of $t$.
Consequently,
\[
(-t)^{-\frac1{p-1}}\|u(\cdot,t)\|_{L^{p+1}(\Omega)}
\leq |\Omega|^{\frac1{p+1}}M_i
\]
for all sufficiently negative $t$.  Lemma
\ref{lm:noninvreasing_norm_p+1} extends this estimate to every $t<0$.

Since $p$ is Sobolev-subcritical, $p+1$ satisfies
\eqref{eq:initial_datum_range_1}.  Applying
\eqref{eq:global_smoothing_effects} on $[2t,t]$ now proves
\eqref{eq:L_infty_up} for every $t<0$, with a constant independent of
$\varepsilon$.}
\end{proof}

{It follows from Lemma \ref{lm:norm_bounds_basic} and Corollary
\ref{cor:weighted-smoothing} that for
$p \in \big(1, \frac{n+1}{(n-1)_+}\big)$, the estimate
\eqref{eq:L_infty_up} holds on $(-\infty,0]$ without any additional
assumptions.}
\begin{lem}\label{lem:Lpplus1-growth}
Let $p \in \big(1, \frac{n+1}{(n-1)_+}\big)$.
If $u$ is a positive ancient solution to \eqref{eq:ADP}, then there exists a constant $C_{\infty}>0$, depending only on $n$,  $p$  and $\Omega$, such that
\begin{equation}\label{eq:poly-bound}
\|u(\cdot,t)\|_{L^{\infty}(\Omega)}
\le C_{\infty}(-t)^{\frac{1}{p-1}},
\quad\forall\,t \in (-\infty,0].
\end{equation}
\end{lem}

\begin{proof}
In the range $p \in \big(1, \frac{n+1}{(n-1)_+}\big)$, the
$L^p_{\Phi_1}$-$L^\infty$ smoothing effect holds.
It follows from Theorem \ref{thm:weighted-smoothing_1} and Lemma
\ref{lm:norm_bounds_basic} that, for every $t<0$,
\begin{align*}
\|u(\cdot,t)\|_{L^\infty(\Omega)}
\le &
\,C_5\,(t-2t)^{-\frac{q}{p-q(p-1)}}
\Big(
\int_\Omega u(x,2t)^p \,\Phi_1(x)\, \ud x
\Big)^{\frac{1}{p-q(p-1)}} \\
\leq & \,C(-t)^{-\frac{q}{p-q(p-1)}}(-2t)^{\frac{p}{(p-1)[p-q(p-1)]}} \\
\leq & \, C_{\infty} (-t)^{\frac{1}{p-1}}.
\end{align*}
The estimate at \(t=0\) follows from the extinction condition.
This completes the proof.
\end{proof}

Next we shall describe the behavior of ancient solutions to \eqref{eq:ADP} near the boundary.
\begin{thm}\label{thm:boundary-behavior}
Suppose that the conditions of Theorem \ref{thm:asymptotic_behaviour} hold. 
Then, for every $\varepsilon>0$ and
$(x,t) \in \Omega\times (-\infty, -\varepsilon]$,
\begin{equation}\label{eq:boundary-behavior}
C^{-1}(-t)^{\frac{1}{p-1}}  \leq \frac{u(x,t)}{\Phi_1(x)} \leq C(-t)^{\frac{1}{p-1}},
\end{equation}
{Here $C\geq1$ may depend on $\varepsilon$ and $u$.}
{Moreover, if
$p \in \big(1,\frac{n+2}{(n-2)_+}\big)$, then
\eqref{eq:boundary-behavior} holds for all $t<0$ with $C$ independent
of $\varepsilon$.}
\end{thm}

\begin{proof}
Fix $\varepsilon>0$ and use the estimates above with $\varepsilon/2$ in place of $\varepsilon$.
We first consider the upper bounds.  For $t\leq-\varepsilon$, by
Lemma \ref{lem:almost-representation}, applied to the time-shifted
solution $u(x,3t+\tau)$ between the shifted times $-t$ and $-2t$,
and the estimate
$\|G_{\Omega}(x, \cdot) \|_{L^1(\Omega)} \leq C \Phi_1(x)$, we have
\begin{align*}
\frac{u(x,t)}{\Phi_1(x)} 
\leq &\,\frac{C p}{p-1} \frac{(-2t)^{1/(p-1)}}{(-2t)^{p/(p-1)}-(-t)^{p/(p-1)}} \|u(\cdot,2t)\|_{L^{\infty}(\Omega)}^p  \\
\leq &\,C \frac{(-2t)^{p/(p-1)}}{-t} = C(-t)^{\frac{1}{p-1}},
\end{align*}
where we have used Lemma \ref{lm:norm_infty_upper_bound} and Lemma
\ref{lem:Lpplus1-growth} in the second inequality.

Now we consider the lower bounds.
Choose $r$ as in condition \textup{(ii)} or \textup{(iii)} of Theorem
\ref{thm:asymptotic_behaviour}; {under condition
\textup{(i)}, fix an exponent $r>p$, depending only on $n$ and $p$,
such that $r+1\geq n(p-1)/2$.}
Arguing as in Lemma \ref{lm:norm_bounds_basic}, for any
$t \in (-\infty, -\varepsilon]$, we have
\begin{align*}
\frac{\ud}{\ud t} \int_\Omega u(x,t)^p\Phi_1(x)\, \mathrm{d}x
= & -\lambda_1 \int_\Omega u(x,t)\Phi_1(x)\, \mathrm{d}x 
=  -\lambda_1 \int_\Omega \frac{u(x,t)^{r+1}}{u(x,t)^{r-1} 
\frac{u(x,t)}{\Phi_1(x)}} \mathrm{d}x \\
\leq & -\lambda_1 \frac{\|u(\cdot, t)\|_{L^{r+1}(\Omega)}^{r+1}}{\|u(\cdot, t)\|_{L^{\infty}(\Omega)}^{r-1}\|\frac{u(\cdot,t)}{\Phi_1}\|_{L^{\infty}(\Omega)}} \\
\leq & -C\frac{(-t)^{\frac{r+1}{p-1}}}{(-t)^{\frac{r-1}{p-1}}(-t)^{\frac{1}{p-1}}}
= -C(-t)^{\frac{1}{p-1}},
\end{align*}
{Here the last inequality follows from Lemmas
\ref{lm:norm_bounds_basic}, \ref{lm:norm_infty_upper_bound}, and
\ref{lem:Lpplus1-growth}, with Lemma \ref{lm:norm_bounds_basic}
applied at exponent $r+1$.}
The estimates used in the preceding differential inequality hold throughout $[t,t/2]$, because they were taken with $\varepsilon/2$.  Integrating from $t$ to $t/2$, for every $t\leq-\varepsilon$ we obtain
\begin{align*}
\int_\Omega u(x,t)^p\Phi_1(x)\, \mathrm{d}x
&\geq \int_\Omega u(x,t/2)^p\Phi_1(x)\, \mathrm{d}x
+C\int_t^{t/2}(-\tau)^{\frac1{p-1}}\,\mathrm d\tau
\geq C(-t)^{\frac{p}{p-1}}.
\end{align*}
Finally, apply Lemma \ref{lm:global-lower-bounds-I} to the shifted solution $u(x,2t+\tau)$, whose shifted extinction time is $-2t$, at the shifted time $-t$.  We obtain
\begin{equation*}
\frac{u(x,t)}{\Phi_1(x)} 
\geq C(-t)^{-1}(-2t)^{\frac{p}{p-1}} = C(-t)^{\frac{1}{p-1}}.
\end{equation*}
{If $p \in \big(1,\frac{n+2}{(n-2)_+}\big)$, the last
estimate in Lemma \ref{lm:norm_infty_upper_bound} allows us to repeat
the argument for every $t<0$.}
This completes the proof.
\end{proof}

By a straightforward computation, we obtain the following estimates.
\begin{thm}\label{thm:boundary_estimate_v}
Let $T >0$.
Suppose that the conditions of Theorem \ref{thm:asymptotic_behaviour} hold. 
Then, for every $s \in (-\infty, T]$,
\begin{equation}\label{eq:boundary-behavior-of-v}
C^{-1}  \leq \frac{v(x,s)}{\Phi_1(x)} \leq C,\quad \forall x \in \Omega.
\end{equation} 
{Here $C>0$ may depend on $T$ and $u$.}
{Furthermore, if
$p \in \big(1,\frac{n+2}{(n-2)_+}\big)$, then
\eqref{eq:boundary-behavior-of-v} holds for all $s\in\mathbb R$ with
a constant independent of $T$.}
\end{thm}

\begin{rem}\label{rmk:Higher_order_regullarity}
Using Theorem \ref{thm:boundary_estimate_v} and the method of
\cite{JX23,JX25}, for every $\ell\in\mathbb Z_{\geq0}$ and
$T\in\mathbb R_-$ we obtain
\begin{equation}\label{eq:unif-window-nonint}
\big\| d^{-1}\partial_s^\ell v \big\|_{L^\infty(\Omega\times[T-1,T])}
+\sup_{s\in[T-1,T]}\big\|\partial_s^\ell v(\cdot,s)\big\|_{C^{3}(\overline\Omega)}
\le C,
\end{equation}
where $d(x)=\operatorname{dist}(x,\partial\Omega)$.
The constant $C$ may depend on $\ell$ and $u$, but is independent of
$T\in\mathbb R_-$.
\end{rem}

From now on, we suppose that all assumptions in Theorem \ref{thm:asymptotic_behaviour} hold unless otherwise stated.
We then prove the energy estimate for solutions to \eqref{eq:Equation_of_v}.
\begin{lem}\label{thm;EnergyEstimate}
Let $ v \in C^{2}(\overline{\Omega} \times (-1, 0]) $ be a positive solution to \eqref{eq:Equation_of_v}.
For  $ -1 < -s_0 < -s_1 < 0 $, there exists a constant $ C $ depending only on $ p $, such that
\begin{equation}\label{equ;EnergyEstimate0-1}
\sup_{s \in (-s_1, 0]} \int_{\Omega} v^{p+1}(x,s) \,\mathrm{d}x + \iint_{Q_1} |\nabla v|^2 \,\ud x\,\ud s \leq \frac{C}{s_0 - s_1} \iint_{Q_0} v^{p+1} \,\ud x\,\ud s,
\end{equation}
where $ Q_i := \Omega \times (-s_i, 0] $ for $ i = 0,1 $.
\end{lem}

\begin{proof}
Let $\psi$ be a smooth cutoff function satisfying $0\leq\psi\leq1$, $\psi=0$ on $(-\infty,-s_0]$, $\psi=1$ on $[-s_1,0]$, and $|\psi'(s)| \le \frac{2}{s_0-s_1}$.
The proof is completed by testing \eqref{eq:Equation_of_v} against $v\psi^2$ and carrying out a routine computation.
\end{proof}

By Theorem \ref{thm:boundary_estimate_v} and the preceding lemma, we
obtain the following corollary.
\begin{cor}\label{cor;energyEstimate}
Let $ v \in C^{2}(\overline{\Omega} \times \mathbb{R}) $ be a positive solution to \eqref{eq:Equation_of_v}.
For every $k \in \mathbb{N}$,
\[
\int_{-k+\frac{1}{2}}^{-k+1}\!\int_\Omega |\nabla v(x,s)|^2 \,\mathrm{d}x\,\mathrm{d}s 
\le C \int_{-k}^{-k+1}\!\int_\Omega v^{p+1}(x,s) \,\mathrm{d}x\,\mathrm{d}s \le C,
\]
where $C=C(n,p,\Omega,C_\infty)>0$ is independent of $k$,
and $C_\infty$ is given by \eqref{eq:L_infty_up} with
$\varepsilon=1$ or by
\eqref{eq:poly-bound}.
\end{cor}

We define the energy functional $F: H^1_0(\Omega)\cap L^{p+1}(\Omega) \to \mathbb{R}$ by
\begin{equation}\label{def:energy-functional}
F[v] = \int_\Omega |\nabla v|^2 \,\mathrm{d}x - \frac{2}{p+1} \int_\Omega v^{p+1} \,\mathrm{d}x.
\end{equation}
It is known that for $p \in \big(1, \frac{n+2}{(n-2)_+}\big)$, all critical points of $F$ are precisely the weak solutions to \eqref{eq:elliptic-equation}.

\begin{lem}\label{lem:energy-monotonicity-bounds}
Let $ v \in C^{2}(\overline{\Omega} \times \mathbb{R}) $ be a positive solution to \eqref{eq:Equation_of_v}. Then $F[v(\cdot, s)]$ is monotonically nonincreasing with respect to $s$.
Furthermore, there exists
$C=C(n,p,\Omega,C_\infty)>0$ such that
$|F[v(\cdot,s)]|\leq C$ for all $s\in\mathbb R_-$, and the limit
$
F_{- \infty} := \lim\limits_{s\to - \infty} F[v(\cdot, s)]
$
exists.
\end{lem}
\begin{proof}
We first establish the monotonicity.
Using integration by parts, we obtain
\begin{equation}\label{eq:energy-derivative-identity}
\begin{split}
\frac{\mathrm{d}}{\mathrm{d}s} F[v(\cdot, s)]
= & \,2 \int_\Omega \big( \nabla v \cdot \nabla v_s - v^p v_s \big) \,\mathrm{d}x \\
= & -\!2 \int_\Omega v_s (\Delta v + v^p) \,\mathrm{d}x
= -2p \int_\Omega v^{p-1} |\partial_s v|^2 \,\mathrm{d}x \leq 0.
\end{split}    
\end{equation}
Thus, $F[v(\cdot, s)]$ is nonincreasing with respect to $s$.

Next, we shall derive the uniform bounds.
The lower bound is easily established by combining the transformation \eqref{eq:v_transform} with Theorem \ref{thm:boundary_estimate_v}.
Precisely, we have
\begin{equation}\label{eq:energy-lower-bound}
F[v(\cdot, s)] \ge - \frac{2}{p+1} \int_\Omega v^{p+1} \,\mathrm{d}x \ge -\frac{C }{p+1}, \qquad \forall\, s \in \mathbb{R}_{-}.
\end{equation}
Then we shall establish the upper bound. 
By the monotonicity of $F$, it suffices to show that there exists a sequence $\{s_k\}$ such that $s_k \to -\infty$ as $k \to \infty$ and
\begin{equation}\label{eq:energy-at-tk}
F[v(\cdot, s_k)] \leq C,\quad \forall\,k \in \mathbb{N}.
\end{equation}
For any $k\in\mathbb{N}$, it follows from Corollary~\ref{cor;energyEstimate} and the mean value theorem that there exists $s_k \in (-k+\tfrac{1}{2}, -k+1]$ such that 
\begin{equation}\label{eq:gradient-at-tk}
\int_\Omega |\nabla v(x,s_k)|^2 \,\mathrm{d}x \le 2C,
\end{equation}
where $C>0$ is a constant independent of $k$.
Clearly, $s_k \to -\infty$ as $k \to \infty$.
By \eqref{eq:gradient-at-tk}, we obtain \eqref{eq:energy-at-tk}.
The existence of $F_{-\infty}$ now follows from the monotonicity and
the lower bound \eqref{eq:energy-lower-bound}.
\end{proof}

\begin{rem}
If $p \in (1, \frac{n+2}{(n-2)_+}\big)$, by a similar argument, the monotonicity and uniform bounds of $F[v(\cdot, s)]$ hold for all $s \in \mathbb{R}$.
As a result, the limit $F_{+\infty}:=\lim\limits_{s\to + \infty} F[v(\cdot, s)]$ exists and is finite.
\end{rem}

\subsection{Asymptotic Behavior}

\begin{lem}\label{lm:sequence_to_-infty}
There exists a sequence $\{s_k\}$ such that $s_k \to -\infty$ as $k \to +\infty$ and
\begin{equation}\label{eq:asymptotic_behavior}
v(\cdot, s_k) \rightarrow v_{-\infty} ~~ \text{in } C^2(\overline{\Omega}) \quad \text{as } k \to +\infty,
\end{equation}
where $v_{-\infty}$ is a positive solution to \eqref{eq:elliptic-equation}.
\end{lem}
\begin{proof}
By Remark \ref{rmk:Higher_order_regullarity} and
\eqref{eq:r_norm_lower_bound}, there exist a nontrivial nonnegative
function $v_{-\infty}$ and a sequence
$\{s_k\}\subset \mathbb{R}$ such that $s_k \to -\infty$ and
$v(\cdot, s_k) \to v_{-\infty}$ in $C^2(\overline{\Omega})$ as $k \to +\infty$.
Now we shall show that $v_{-\infty}$ is a positive solution to \eqref{eq:elliptic-equation}.
Using \eqref{eq:energy-derivative-identity} and the Cauchy-Schwarz inequality, we obtain that for any $s \in \mathbb{R}_{-}$,
\begin{equation*}
\begin{split}
\int_{\Omega} |v^{\frac{p+1}{2}}(x, s_k+s)-v^{\frac{p+1}{2}}(x, s_k)|^2 \,\ud x 
= & \int_{\Omega} \Big(\int_{s_k+s}^{s_k}\partial_s v^{\frac{p+1}{2}}\,\ud s\Big)^2\,\ud x \\
\leq & \int_{\Omega} (-s) \int_{s_k+s}^{s_k}  |\partial_s v^{\frac{p+1}{2}}|^2 \,\ud s\,\ud x \\
= & \,\frac{(p+1)^2}{4}\int_{\Omega} (-s) \int_{s_k+s}^{s_k}  |v^{\frac{p-1}{2}}\partial_s v|^2 \,\ud s\,\ud x \\
= & \,\frac{(p+1)^2(-s)}{8p}\Big(F[v(\cdot, s_k+s)]-F[v(\cdot, s_k)]\Big).
\end{split}
\end{equation*}
Then by Lemma \ref{lem:energy-monotonicity-bounds}, we have
\begin{equation*}
\int_{\Omega} |v^{\frac{p+1}{2}}(\cdot, s_k+s)-v_{-\infty}^{\frac{p+1}{2}}|^2 \,\ud x \rightarrow 0
\end{equation*}
locally uniformly in $s$ as $k \to +\infty$.
By using the inequality $|a-b|^{\frac{p+1}{2}} \leq |a^{\frac{p+1}{2}}-b^{\frac{p+1}{2}}|$ for $a,b \geq 0$, we get
\begin{equation*}
v(\cdot, s_k+s) \rightarrow v_{-\infty} ~~ \text{in } L^{p+1}(\Omega)
\end{equation*}
locally uniformly in $s$ as $k \to +\infty$.
It follows from Remark \ref{rmk:Higher_order_regullarity}
and the interpolation inequality on page 126 of Nirenberg \cite{N59} that
\begin{equation*}
v(\cdot, s_k+s) \rightarrow v_{-\infty} ~~ \text{in } C^{2}(\overline{\Omega})
\end{equation*}
locally uniformly in $s$ as $k \to +\infty$.
Testing the equation \eqref{eq:Equation_of_v} on $\Omega \times [s_k-1, s_k]$ by $\phi \in C_c^\infty(\Omega)$, we have
\begin{equation*}
\begin{split}
&\int_\Omega \Big(v^p(\cdot, s_k) - v^p(\cdot, s_{k}-1) \Big)\phi \,\mathrm{d}x \\
= &  -\int_{-1}^{0} \int_\Omega \nabla v(\cdot, s_k + s) \cdot \nabla \phi \,\mathrm{d}x\,\mathrm{d}s + \int_{-1}^{0} \int_\Omega v^p(\cdot, s_k + s)\,\phi \,\mathrm{d}x\,\mathrm{d}s \\
= &-\int_\Omega \nabla v(\cdot, s_k + \bar{s}) \cdot \nabla \phi \,\mathrm{d} x+ \int_\Omega v^p(\cdot, s_k + \bar{s})\,\phi \,\mathrm{d}x,
\end{split}
\end{equation*}
where $\bar{s} \in [-1,0]$ is given by the mean value theorem.
Letting $k \to +\infty$, we obtain
\begin{equation*}
0 = -\int_\Omega \nabla v_{-\infty} \cdot \nabla \phi \,\mathrm{d}x + \int_\Omega v_{-\infty}^p \phi \,\mathrm{d}x,
\end{equation*}
Thus, $v_{-\infty}$ is a positive solution to \eqref{eq:elliptic-equation}.
\end{proof}
Taking the limit in \eqref{eq:boundary-behavior-of-v} along the sequence from Lemma \ref{lm:sequence_to_-infty}, we find that the profile $S:=v_{-\infty}$ satisfies
\begin{equation}\label{eq:HI_elliptic}
C^{-1} \Phi_1(x) \leq S(x) \leq C \Phi_1(x),\quad \forall x \in \Omega.
\end{equation}
Here and below, the constant may depend on $n,p,\Omega$ and the corresponding a priori bound $C_\infty$.
Define $h(x,s):=v(x,s)-S(x)$, then $h$ satisfies
\begin{equation}\label{eq:equ_h}
\begin{cases}
a(x,s)S^{p-1}\partial_s h = \Delta h + c(x, s)S^{p-1}h &\text{~in~} \Omega \times \mathbb{R}, \\
h = 0 &\text{~on~} \partial \Omega \times \mathbb{R},
\end{cases}
\end{equation}
where 
$$a(x,s)= p\left(\frac{v(x,s)}{S(x)}\right)^{p-1}, \quad c(x,s) := \int_0^1 p\Big(\frac{v(x,s)}{S(x)}+\lambda \Big(1-\frac{v(x,s)}{S(x)}\Big)\Big)^{p-1} \,\mathrm{d}\lambda.$$
Using \eqref{eq:HI_elliptic} and Theorem \ref{thm:boundary_estimate_v}, there exists a constant $C=C(n,p,\Omega,C_\infty)\geq 1$ such that
\begin{equation}\label{eq:bounds_a_c}
C^{-1}  \leq |{a(x,s)}|,\,|{c(x,s)}| \leq C ,\quad \forall x \in \Omega,~\forall s \in (-\infty, 1].
\end{equation}
\begin{lem}\label{lm:L2_weighted_estimate}
Let $S$ be a bounded positive solution to \eqref{eq:elliptic-equation}, and let $h$ satisfy \eqref{eq:equ_h}.
Then for any $s \in \mathbb{R}_{-}$ and $s_0 \in [0,1]$,
\begin{equation*}
\int_{\Omega} h^2(\cdot, s + s_0) S^{p-1} \,\mathrm{d}x \leq Ce^{Cs_0} \int_{\Omega} h^2(\cdot, s) S^{p-1} \,\mathrm{d}x,
\end{equation*}
where $C$ is a positive constant depending only on $n$,
$p$, $\Omega$, and $C_\infty$.
\end{lem}
\begin{proof}
Using \eqref{eq:equ_h} and integration by parts, for any $s \in \mathbb{R}_{-}$ and $s_0 \in [0,1]$, we have
\begin{equation*}
\begin{split}
\frac{\ud }{\ud s}\int_{\Omega} a(x, s)h^2 S^{p-1} \,\mathrm{d}x    
= & -\!2\int_{\Omega} |\nabla h|^2 \,\mathrm{d}x + \int_{\Omega} (2c(x, s) + \partial_s a(x, s)) h^2 S^{p-1} \,\mathrm{d}x\\
\leq & \,C \int_{\Omega} a(x,s)h^2 S^{p-1} \,\mathrm{d}x,
\end{split}
\end{equation*}
where $C$ depends on the constant in \eqref{eq:bounds_a_c} and $\|\partial_s a\|_{L^\infty(\Omega \times [s, s+s_0])}$.
It follows from \eqref{eq:HI_elliptic},
\eqref{eq:bounds_a_c}, and Remark
\ref{rmk:Higher_order_regullarity} that
$\|\partial_s a\|_{L^\infty(\Omega\times[s,s+s_0])}$ is bounded
uniformly for $s\in\mathbb R_-$ and $s_0\in[0,1]$ by a constant
depending only on $n$, $p$, $\Omega$, and $C_\infty$.
Integrating the above inequality from $s$ to $s+s_0$, we have
\begin{equation*}
\int_{\Omega} a(x, s+s_0)h^2(\cdot, s+s_0) S^{p-1} \,\mathrm{d}x \leq e^{Cs_0} \int_{\Omega} a(x, s)h^2(\cdot, s) S^{p-1} \,\mathrm{d}x,
\end{equation*}
and hence the desired estimate follows from \eqref{eq:bounds_a_c}.
\end{proof}
By the preceding lemma and Theorem 2.19 in \cite{JX23}, we obtain the
following estimate.
\begin{lem}\label{lm:C3_estimate} 
Let $S$ be a bounded positive solution to \eqref{eq:elliptic-equation}, and let $h$ satisfy \eqref{eq:equ_h}.
Then for any $T \in \mathbb{R}_{-}$,
\begin{equation*}
\sup_{s\in[T+\frac12,\,T+1]}\| \partial_s h(\cdot, s)\|_{C^3(\overline{\Omega})}
+ \sup_{s\in[T+\frac12,\,T+1]}\| h(\cdot, s)\|_{C^3(\overline{\Omega})}
\leq C \|h(\cdot, T)\|_{L^2(\Omega;\,\Phi_1^{p-1})}.
\end{equation*}
{Here, $C>0$ is a constant depending only on $n$, $p$,
$\Omega$, and $C_\infty$.}
\end{lem}

Set
\[
\mathscr{D}
:=
\left\{
w\in C^2(\overline{\Omega}) :
0<\inf_{\Omega}\frac{w}{\Phi_1}
\leq
\sup_{\Omega}\frac{w}{\Phi_1}
<\infty
\right\}.
\]
{By \eqref{eq:HI_elliptic}, we have \(S\in\mathscr{D}\).
Proposition 6.1 in Feireisl--Simondon \cite{FS00}, which is a
positive-solution version of the {\L}ojasiewicz--Simon inequality
of Simon \cite{S83}, applies at \(S\) for every \(p>1\), including noninteger
values of \(p\). The energy
functional used there is \(\mathcal E=\frac12F\).

Consequently, there exist $\theta\in(0,1/2]$, $ \delta_0>0$, $ C>0$ such that, for every \(w\in\mathscr{D}\) satisfying $\|w-S\|_{C^2(\overline{\Omega})}<\delta_0$, one has
\begin{equation}\label{eq:LS}
\begin{split}
\bigl|F[w]-F[S]\bigr|^{1-\theta}
&\leq
C\|-\Delta w-w^p\|_{H^{-1}(\Omega)}
\\
&\leq
C\|-\Delta w-w^p\|_{L^2(\Omega)}.
\end{split}
\end{equation}
 We have stated the
local neighborhood in the stronger \(C^2(\overline{\Omega})\)
topology, which is sufficient for the application below.

The following backward-time version of the finite-length argument in
the proof of Theorem 3.1 in Feireisl--Simondon \cite{FS00} upgrades the subsequential
convergence in Lemma \ref{lm:sequence_to_-infty} to full convergence.}

\begin{lem}\label{lm:convergence_to_-infty}
Let $v_{-\infty}$ be the positive solution to \eqref{eq:elliptic-equation} given in Lemma \ref{lm:sequence_to_-infty}.
Then we have
\begin{equation*}
v(\cdot, s) \rightarrow v_{-\infty} ~~ \text{in } C^{2}(\overline{\Omega}) \quad \text{as } s \to -\infty.
\end{equation*}
\end{lem}
\begin{lem}\label{lm:relative_error_decay}
Let $v_{-\infty}$ be the positive solution to \eqref{eq:elliptic-equation} given in Lemma \ref{lm:sequence_to_-infty}.
Then there exists $s_0 \in (-\infty, 0)$ such that, for every
$s \in (-\infty, s_0]$,
\begin{equation*}
\Big\| \frac{v(\cdot,s)}{v_{-\infty}} - 1\Big\|_{C^2(\overline{\Omega})} \leq \begin{cases}
C(-s)^{-\gamma}, &\text{if}~\theta \in (0, {1}/{2}), \\
Ce^{\gamma s}, &\text{if}~\theta = {1}/{2},
\end{cases}
\end{equation*}
where $\theta \in (0, {1}/{2}]$ is given in the {\L}ojasiewicz-Simon type inequality \eqref{eq:LS} and $\gamma > 0$ is a constant depending only on $n$, $p$, $\Omega$, and $v_{-\infty}$.
\end{lem}
\begin{proof}
Using \eqref{eq:energy-derivative-identity} and Remark \ref{rmk:Higher_order_regullarity}, we have that for any $s \in \mathbb{R}_{-}$,
\begin{align*}
\frac{\ud}{\ud s} \big(F[v_{-\infty}]-F[v(\cdot, s)]\big) 
= & \,\frac{2}{p}\int_\Omega v^{1-p}|\partial_s v^{p}|^2 \,\mathrm{d}x \\
\geq & \,C\int_\Omega |\partial_s v^{p}|^2 \,\mathrm{d}x = C\|-\!\Delta v(\cdot, s) - v^p(\cdot, s)\|_{L^2(\Omega)}^2.
\end{align*}
Combining the above inequality with the {\L}ojasiewicz-Simon type inequality \eqref{eq:LS}, there exists $s_0 \in \mathbb{R}_{-}$ such that for any $s \leq s_0$,
\begin{equation*}
\frac{\ud}{\ud s} \big(F[v_{-\infty}]-F[v(\cdot, s)]\big) \geq C \big(F[v_{-\infty}]-F[v(\cdot, s)]\big)^{2(1-\theta)}.
\end{equation*}
Thus for all $s \leq s_0$,
\begin{equation*}
F[v_{-\infty}]-F[v(\cdot, s)] \leq 
\begin{cases}
C(-s)^{-\frac{1}{1-2\theta}}, &\text{if}~\theta \in (0, {1}/{2}), \\
e^{Cs}, &\text{if}~\theta = {1}/{2}.
\end{cases}
\end{equation*}
Arguing as in the proof of Theorem 1.2 in \cite{JXY24}, we obtain, for
all $s \leq s_0$,
\begin{equation*}
\|v(\cdot, s)-v_{-\infty}\|_{L^{p+1}(\Omega)} \leq 
\begin{cases}
C(-s)^{-\gamma}, &\text{if}~\theta \in (0, {1}/{2}), \\
Ce^{\gamma s}, &\text{if}~\theta = {1}/{2}.
\end{cases}
\end{equation*}
By Remark \ref{rmk:Higher_order_regullarity} and the interpolation
inequality on page 126 of Nirenberg \cite{N59}, we obtain
\begin{equation*}
\|v(\cdot, s)-v_{-\infty}\|_{C^1(\overline{\Omega})} \leq \begin{cases}
C(-s)^{-\gamma}, &\text{if}~\theta \in (0, {1}/{2}), \\
Ce^{\gamma s}, &\text{if}~\theta = {1}/{2},
\end{cases}\quad \forall s \leq s_0.
\end{equation*}
Since $v(\cdot, s) \equiv v_{-\infty}\equiv 0$ on $\partial \Omega$, we have for all $x \in \Omega$ and all $s \leq s_0$ that
\begin{equation*}
\begin{split}
\Big| \frac{v(x,s)-v_{-\infty}(x)}{v_{-\infty}(x)}\Big| 
\leq &\,\frac{\|v(\cdot, s)-v_{-\infty}\|_{C^1(\overline{\Omega})}d(x)}{d(x)/C}\\
\leq &
\begin{cases}
C(-s)^{-\gamma}, &\text{if}~\theta \in (0, {1}/{2}), \\
Ce^{\gamma s}, &\text{if}~\theta = {1}/{2}.
\end{cases}
\end{split}
\end{equation*}
It follows that
\begin{equation*}
\Big\| \frac{v(\cdot,s)}{v_{-\infty}} - 1\Big\|_{L^\infty(\Omega)} 
\leq \begin{cases}
C(-s)^{-\gamma}, &\text{if}~\theta \in (0, {1}/{2}), \\
Ce^{\gamma s}, &\text{if}~\theta = {1}/{2},
\end{cases}
\quad \forall s \leq s_0.
\end{equation*}
Similarly, we obtain
\begin{equation*}
\Big\| \frac{v(\cdot,s)}{v_{-\infty}} - 1\Big\|_{C^2(\overline{\Omega})} \leq \begin{cases}
C(-s)^{-\gamma}, &\text{if}~\theta \in (0, {1}/{2}), \\
Ce^{\gamma s}, &\text{if}~\theta = {1}/{2},
\end{cases}
\quad \forall s \leq s_0,
\end{equation*}
with a possibly different constant $\gamma>0$.
\end{proof}

\subsection{Sharp Rates}
We shall follow the approach in \cite{CMS23} to establish the sharp convergence rates of the relative error term $g$, which is defined by
$$g(x,s):=\frac{v(x,s)-v_{-\infty}(x)}{v_{-\infty}(x)},$$
which satisfies
\begin{equation} 
p\partial_s g + \mathcal{A} g = \mathcal{N}(g) \quad\text{in }\Omega\times\mathbb{R},
\end{equation}
where the linear operator $\mathcal{A}$ admits the equivalent expressions
\begin{align*}
\mathcal{A}g
&= -v_{-\infty}^{-p}\Delta(v_{-\infty}g) - pg,
\end{align*}
and the nonlinear term is given by
$$
\mathcal{N}(g) = (1+g)^p - pg - 1 + p\bigl(1-(1+g)^{p-1}\bigr)\partial_s g.
$$

For $q\in[1,\infty)$ and $\sigma\in\mathbb{R}_+$, define the weighted Lebesgue space
$$
L^q_{\sigma}(\Omega)
= \Bigl\{ g\in L^1_{\mathrm{loc}}(\Omega) :
\|g\|_{L^q_{\sigma}} = \Bigl(\int_{\Omega}|g|^q v_{-\infty}^{\sigma}\,dx\Bigr)^{1/q}<\infty \Bigr\}.
$$
Then $L^2_{\sigma}(\Omega)$ is a Hilbert space with inner product
$$
\langle g_1,g_2\rangle_{L^2_{\sigma}(\Omega)}
= \int_{\Omega}g_1g_2 v_{-\infty}^{\sigma}\,dx,\quad \forall g_1,g_2\in L^2_{\sigma}(\Omega).
$$
By Lemma 2.1 in Bonforte--Figalli \cite{BF21}, the eigenvalue problem
$$
\begin{cases}
(-\Delta - pv_{-\infty}^{p-1})\tilde{e} = \lambda v_{-\infty}^{p-1}\tilde{e} & \text{in }\Omega,\\
\tilde{e}=0 & \text{on }\partial\Omega
\end{cases}
$$
admits eigenpairs $\{(\lambda_i,\tilde{e}_i)\}_{i=-I}^\infty$ such that
\begin{itemize}
\item the negative eigenvalues satisfy
$$
1-p=\lambda_{-I}<\lambda_{-I+1}\le\cdots\le\lambda_{-1}<0.
$$
Writing $K:=\dim\ker\mathcal L_{v_{-\infty}}$, if $K\geq1$, then
$$
\lambda_0=\cdots=\lambda_{K-1}=0<\lambda_K\le\lambda_{K+1}\le\cdots\to+\infty;
$$
if $K=0$, the zero block is absent and
$$
0<\lambda_0\le\lambda_1\le\cdots\to+\infty;
$$
\item $\{\tilde{e}_i\}_{i=-I}^\infty$ is an orthonormal basis of $L^2_{p-1}(\Omega)$.
\end{itemize}
Since $g\mapsto gv_{-\infty}^{-1}$ is an isomorphism from $L^2_{p-1}(\Omega)$ to $L^2_{p+1}(\Omega)$,
the family $\{(\lambda_i,e_i=\tilde{e}_i v_{-\infty}^{-1})\}_{i=-I}^\infty$ forms an orthonormal basis of $L^2_{p+1}(\Omega)$.
We have the direct decomposition
$$
L^2_{p+1}(\Omega) = \mathscr{X}_- \oplus \mathscr{X}_0 \oplus \mathscr{X}_+,
$$
where
$$
\mathscr{X}_-=\mathrm{span}\{e_i:i<0\},\quad
\mathscr{X}_0=\mathrm{span}\{e_i:0\le i<K\},\quad
\mathscr{X}_+=\overline{\mathrm{span}\{e_i:i\ge K\}}.
$$
Let $\mathcal{P}_-,\mathcal{P}_0,\mathcal{P}_+$ be the corresponding orthogonal projections. Then
$$
g = \mathcal{P}_-g + \mathcal{P}_0g + \mathcal{P}_+g =: g_-+g_0+g_+,\quad \forall g\in L^2_{p+1}(\Omega).
$$
For brevity, in the rest of this subsection we write
$\|\cdot\|=\|\cdot\|_{L^2_{p+1}(\Omega)}$ and
$\langle\cdot,\cdot\rangle=\langle\cdot,\cdot\rangle_{L^2_{p+1}(\Omega)}$.
We now prove Theorem \ref{thm:asymptotic_behaviour}.
\begin{proof}[Proof of Theorem \ref{thm:asymptotic_behaviour}]
It follows from Lemma \ref{lm:relative_error_decay} that the upper bound in \eqref{eq:poly_decay} holds.
If $v_{-\infty}$ is integrable, then by Lemma 1 in Adams--Simon
\cite{AS88}, the
{\L}ojasiewicz-Simon type inequality \eqref{eq:LS} holds with
$\theta = 1/2$.
It follows from Lemma \ref{lm:relative_error_decay} that \eqref{eq:exp_decay} holds.
Next we shall prove the lower bound in \eqref{eq:poly_decay} and a sharp version of \eqref{eq:exp_decay}.

For $a = -$, $0$ or $+$, we have
\begin{align*}
2\|g_a\|\frac{\ud }{\ud s}\|g_a\|= \frac{\ud }{\ud s}\|g_a\|^2
=  2\langle g_a,\partial_s g_a \rangle 
=  -\frac{2}{p}\big\langle g_a,\mathcal{A} g_a \big\rangle + \frac{2}{p}\big\langle g_a, \mathcal{P}_a\mathcal{N}(g) \big\rangle.
\end{align*}
By the Cauchy--Schwarz inequality, we have
\begin{equation*}
\big \vert \big\langle g_a,\mathcal{A} g_a \big\rangle \big\vert 
\leq \Vert g_a\Vert  \Vert \mathcal{A} g_a \Vert,\quad 
\big \vert \big\langle g_a, \mathcal{P}_a\mathcal{N}(g) \big\rangle \big\vert \leq \Vert g_a\Vert  \Vert \mathcal{N}(g) \Vert.
\end{equation*}
Consequently, we have 
\begin{equation}\label{eq:projections_ODE_I}
\begin{split}
\frac{\ud }{\ud s}\|g_-\| \geq & -\!\frac{\lambda_{-1}}{p}\|g_-\| - \frac{1}{p}\big\Vert \mathcal{N}(g) \big\Vert, \\
\Big|\frac{\ud }{\ud s}\|g_0\|\Big| \leq & \,\frac{1}{p}\big\Vert \mathcal{N}(g) \big\Vert, \\
\frac{\ud }{\ud s}\|g_+\| \leq & -\!\frac{\lambda_K }{p}\|g_+\| + \frac{1}{p}\big\Vert \mathcal{N}(g) \big\Vert.
\end{split}
\end{equation}
It follows from Lemma \ref{lm:relative_error_decay} that $\|g(\cdot, s)\|_{C^2(\overline{\Omega})} \to 0$ as $s \to -\infty$.
{
Set $h=v-v_{-\infty}=v_{-\infty}g$.  Both $h$ and $\partial_sh$
vanish on $\partial\Omega$.  Taking $T=s-1$ in Lemma
\ref{lm:C3_estimate}, and using
$v_{-\infty}\asymp\Phi_1\asymp d(\cdot,\partial\Omega)$, we obtain
\begin{equation}\label{eq:relative-Linfty-smoothing}
\begin{split}
\|g(\cdot,s)\|_{L^\infty(\Omega)}
&+\|\partial_sg(\cdot,s)\|_{L^\infty(\Omega)}\\
&\leq C\bigl(
\|h(\cdot,s)\|_{C^1(\overline\Omega)}
+\|\partial_sh(\cdot,s)\|_{C^1(\overline\Omega)}
\bigr)\\
&\leq C\|h(\cdot,s-1)\|_{L^2(\Omega;\,\Phi_1^{p-1})}
\leq C\|g(\cdot,s-1)\|,
\qquad s<0.
\end{split}
\end{equation}
Here the first inequality follows from the mean value theorem and the
boundary conditions, while the last one follows from
$h=v_{-\infty}g$ and $v_{-\infty}\asymp\Phi_1$.

For all sufficiently negative $s$, Taylor's formula applied to the
definition of $\mathcal N$ gives
\begin{equation}\label{eq:nonlinear_term_estimate}
|\mathcal N(g(x,s))|
\leq C|g(x,s)|\bigl(|g(x,s)|+|\partial_sg(x,s)|\bigr).
\end{equation}
Taking the $L^2_{p+1}(\Omega)$ norm in
\eqref{eq:nonlinear_term_estimate}, we obtain
\[
\|\mathcal N(g(\cdot,s))\|
\leq C\bigl(
\|g(\cdot,s)\|_{L^\infty(\Omega)}
+\|\partial_sg(\cdot,s)\|_{L^\infty(\Omega)}
\bigr)\|g(\cdot,s)\|.
\]
Since $\|g(\cdot,s-1)\|\to0$ as $s\to-\infty$,
\eqref{eq:relative-Linfty-smoothing} shows that, for every
$\varepsilon>0$, there exists $s_0=s_0(\varepsilon)<0$ such that
\begin{equation}\label{eq:nonlinear_term_estimate_norm}
\|\mathcal N(g(\cdot,s))\|\leq\varepsilon\|g(\cdot,s)\|,
\qquad s\leq s_0.
\end{equation}
}
Set $\lambda:= 2p^{-1}\min\{\lambda_K, -\lambda_{-1}\}$, and $\tau = \lambda s$.
Then there exists $\tau_0=\tau_0(\varepsilon) < 0$ such that for all $\tau \leq \tau_0$, $\|g(\cdot,\tau)\|\rightarrow 0$ as $\tau \to -\infty$ and
\begin{equation}\label{eq:projections_ODE_II}
\begin{split}
\frac{\ud }{\ud \tau}\|g_-\| \geq & -\!\frac{\lambda_{-1}}{p\lambda}\|g_-\| - \frac{\varepsilon}{p\lambda}\big\Vert g \big\Vert, \\
\Big|\frac{\ud }{\ud \tau}\|g_0\|\Big| \leq & \,\frac{\varepsilon}{p\lambda}\big\Vert g \big\Vert, \\
\frac{\ud }{\ud \tau}\|g_+\| \leq & -\!\frac{\lambda_K }{p\lambda}\|g_+\| + \frac{\varepsilon}{p\lambda}\big\Vert g \big\Vert.
\end{split}
\end{equation}
From now on, $s_0<0$ is a constant that can be fixed to different values as needed.
By Lemma 4.6 in \cite{CHH22}, which refines Lemma A.1 in
Merle--Zaag \cite{MZ98}, either $g\equiv0$, or one of the following two
alternatives holds:
\begin{equation}\label{eq:caseI}
\|g_-\| + \|g_+\| = o(\|g_0\|)~~\text{as}~s \to -\infty
\end{equation}
or
\begin{equation}\label{eq:caseII}
\|g_0\| + \|g_+\| \leq \frac{100\varepsilon}{p\lambda}\|g_-\|~~\text{for}~s \leq s_0
\end{equation}
The case $g\equiv0$ is immediate.  We therefore investigate the two
alternatives separately.

\textbf{Case I.} If \eqref{eq:caseI} holds, then we can suppose that
\begin{equation*}
\|g_-\| + \|g_+\| \leq \frac{1}{2}\|g_0\|~~ \text{for}~s \leq s_0.
\end{equation*}
Thus using \eqref{eq:projections_ODE_I}, we have
\begin{equation*}
\left|\frac{\ud }{\ud s}\|g_0\|\right| \leq  \frac{\varepsilon}{p}\big\Vert g \big\Vert \leq \frac{2\varepsilon}{p}\|g_0\|.
\end{equation*}
By the Gronwall inequality, for all $s \leq s_0$ we have
\begin{equation*}
\|g(\cdot, s-1)\|\leq 2\|g_0(\cdot, s-1)\| \leq C\|g_0(\cdot, s)\|\leq C\|g(\cdot, s)\|.
\end{equation*}
{Combining \eqref{eq:relative-Linfty-smoothing},
\eqref{eq:nonlinear_term_estimate}, and the preceding comparison of
$\|g(\cdot,s-1)\|$ with $\|g(\cdot,s)\|$, we obtain, for $s\leq s_0$,}
\begin{equation}\label{eq:estimate_nonlinear}
{
\begin{aligned}
\|\mathcal{N}(g(\cdot,s))\|
&\leq C\bigl(\|g(\cdot,s)\|_{L^\infty}
+\|\partial_sg(\cdot,s)\|_{L^\infty}\bigr)\|g(\cdot,s)\|\\
&\leq C\|g(\cdot,s-1)\|\|g(\cdot,s)\|
\leq C\|g(\cdot,s)\|^2.
\end{aligned}
}
\end{equation}
By \eqref{eq:projections_ODE_I} again, we have
\begin{equation*}
\frac{\ud }{\ud s}\|g_0\| \leq  \frac{C}{p}\|g\|^2 \leq C\|g_0\|^2,\quad \forall s \leq s_0,
\end{equation*}
and so 
\begin{equation*}
\|g(\cdot, s)\|\geq \|g_0(\cdot, s)\| \geq C(-s)^{-1},\quad \forall s \leq s_0.
\end{equation*}
{This} proves the lower bound in \eqref{eq:poly_decay}.

\textbf{Case II.} If \eqref{eq:caseII} holds, then
\begin{equation*}
\|g_0\| + \|g_+\| \leq \frac{1}{2}\|g_-\|,\quad \forall s \leq s_0.
\end{equation*}
Combining the above inequality with \eqref{eq:nonlinear_term_estimate_norm}, we have
\begin{equation*}
\|\mathcal{N}(g)\| \leq 2\varepsilon \| g_{-}\|,\quad \forall s \leq s_0.
\end{equation*}
It follows from \eqref{eq:projections_ODE_I} that
\begin{equation*}
\frac{\ud }{\ud s}\|g_-\| \geq  \frac{-\lambda_{-1}-2\varepsilon}{p}\|g_-\|, \quad \forall s \leq s_0,
\end{equation*}
and by the Gronwall inequality, we have
\begin{equation*}
\|g(\cdot, s)\| \leq 2\|g_-(\cdot, s)\| \leq C e^{\frac{-\lambda_{-1}-2\varepsilon}{p}(s-s_0)}\|g_-(\cdot, s_0)\|,\quad \forall s \leq s_0.
\end{equation*}
After decreasing $\varepsilon$ if necessary, set $a:=(-\lambda_{-1}-2\varepsilon)/p>0$ and arrange in addition that $2a>-\lambda_{-1}/p$. Integrating the preceding differential inequality over $[s-1,s]$ and using the unstable-mode dominance in Case II, we obtain
\[
\|g(\cdot,s-1)\|\leq2\|g_-(\cdot,s-1)\|
\leq2e^{-a}\|g_-(\cdot,s)\|\leq C\|g(\cdot,s)\|.
\]
{It follows from \eqref{eq:relative-Linfty-smoothing} that
\[
\|g(\cdot,s)\|_{L^\infty}
+\|\partial_sg(\cdot,s)\|_{L^\infty}
\leq C\|g(\cdot,s-1)\|
\leq C\|g(\cdot,s)\|.
\]
Combining this with \eqref{eq:nonlinear_term_estimate}, for
$s\leq s_0$ we have}
{
\begin{equation*}
\begin{aligned}
\|\mathcal{N}(g(\cdot,s))\|
&\leq C\bigl(\|g(\cdot,s)\|_{L^\infty}
+\|\partial_sg(\cdot,s)\|_{L^\infty}\bigr)\|g(\cdot,s)\|
\leq C\|g(\cdot,s)\|^2\\
&\leq C e^{\frac{-\lambda_{-1}-2\varepsilon}{p}(2s-2s_0)}
\|g_-(\cdot,s_0)\|^2.
\end{aligned}
\end{equation*}
}
Combining the above inequality with \eqref{eq:projections_ODE_I}, we have 
\begin{equation*}
\frac{\ud }{\ud s}\|g_-(\cdot, s)\| \geq  -\frac{\lambda_{-1}}{p}\|g_-(\cdot, s)\| - C e^{\frac{-\lambda_{-1}-2\varepsilon}{p}(2s-2s_0)}\|g_-(\cdot, s_0)\|^2,\quad \forall s \leq s_0.
\end{equation*}
The choice $2a>-\lambda_{-1}/p$ makes the forcing term integrable after multiplication by the corresponding integrating factor. By the Gronwall inequality again, we have
\begin{equation*}
\|g_-(\cdot, s)\| \leq C e^{\frac{-\lambda_{-1}}{p}s}\|g_-(\cdot, s_0)\|,\quad \forall s \leq s_0,
\end{equation*}
and so 
\begin{equation*}
\|g(\cdot, s)\| \leq 2\|g_-(\cdot, s)\| \leq C e^{\frac{-\lambda_{-1}}{p}s}\|g_-(\cdot, s_0)\|\leq C e^{\frac{-\lambda_{-1}}{p}s}\|g(\cdot, s_0)\|,\quad \forall s \leq s_0.
\end{equation*}
Since $\|g\| = \|h\|_{L^2_{p-1}(\Omega)}$, we get
\begin{equation*}
\|h(\cdot, s)\|_{L^2_{p-1}(\Omega)} \leq C e^{\frac{-\lambda_{-1}}{p}s},\quad \forall s \leq s_0.
\end{equation*}
Using Lemma \ref{lm:C3_estimate}, we have
\begin{equation*}
\|h(\cdot, s)\|_{C^3(\overline{\Omega})}
\leq C\|h(\cdot, s-1)\|_{L^2_{p-1}(\Omega)}
\leq C e^{\frac{-\lambda_{-1}}{p}s},\quad \forall s \leq s_0.
\end{equation*}
Consequently, for all $s \leq s_0$,
\begin{equation*}
\|g(\cdot, s)\|_{C^2(\overline{\Omega})} \leq C\|h(\cdot, s)\|_{C^3(\overline{\Omega})} \leq C e^{\frac{-\lambda_{-1}}{p}s},
\end{equation*}
which is the sharp version of \eqref{eq:exp_decay}.
{After enlarging the trajectory-dependent constant $C$, the relevant
estimate extends from $s\leq s_0$ to the fixed interval $s<-1$.
In the polynomial alternative, a nonstationary trajectory cannot meet
$v_{-\infty}$ at a finite time: otherwise
\eqref{eq:energy-derivative-identity}, together with
$F[v(\cdot,s)]\to F[v_{-\infty}]$ as $s\to-\infty$, would force the
trajectory to be stationary.  Hence the lower bound also extends across
the compact interval $[s_0,-1]$ after changing $C$.}
This completes the proof of Theorem \ref{thm:asymptotic_behaviour}.
\end{proof}

{Finally, we prove Theorem \ref{thm:ancient-rigidity}.}

\begin{proof}[Proof of Theorem \ref{thm:ancient-rigidity}]
We first assume condition \textup{(i)} and set $S=v_{-\infty}$.  If
$g=v/S-1$ vanishes identically, there is nothing to prove.  Otherwise,
by integrability, \eqref{eq:exp_decay} holds, so the center-dominant alternative
\eqref{eq:caseI}, which has the lower bound $\|g(\cdot,s)\|\geq
C(-s)^{-1}$, is impossible.  Hence the unstable component dominates and,
for all sufficiently negative $s$,
\begin{equation}\label{eq:index-one-dominance}
\|g(\cdot,s)\|\leq 2\|g_-(\cdot,s)\|,
\qquad
\|\mathcal N(g(\cdot,s))\|\leq C\|g(\cdot,s)\|^2.
\end{equation}

We first relate the usual Morse index to the weighted spectral decomposition
used above.  Since $S>0$ in $\Omega$, for every
$0\neq\phi\in H_0^1(\Omega)$ the denominator in the
weighted Rayleigh quotient
\[
\frac{\mathcal Q_S(\phi)}{\displaystyle\int_\Omega
S^{p-1}\phi^2\,dx}
\]
is positive.  By the min--max characterization of the weighted
eigenvalues, the number of
negative eigenvalues of the weighted problem, counted with multiplicity,
equals the Morse index of $S$ and is thus one.  Moreover,
\[
\mathcal L_S S=(1-p)S^{p-1}S.
\]
Consequently, its negative eigenspace is $\operatorname{span}\{S\}$.  In the
relative variables, this becomes
\[
\mathcal A\mathbf 1=(1-p)\mathbf 1,
\qquad \mathscr X_- =\operatorname{span}\{\mathbf 1\}.
\]
Write
$g_-(\cdot,s)=a(s)\mathbf 1$.  Projecting the equation for $g$ onto
$\mathscr X_-$ and using \eqref{eq:index-one-dominance}, we obtain
\begin{equation}\label{eq:index-one-amplitude}
a'(s)-\frac{p-1}{p}a(s)=O(a(s)^2).
\end{equation}
If $a(s_*)=0$ at some sufficiently negative time, then
by \eqref{eq:index-one-dominance}, $g(\cdot,s_*)=0$.  Since
$F[v(\cdot,s)]\to F[S]$ as $s\to-\infty$, it follows from the energy
identity \eqref{eq:energy-derivative-identity} that $v\equiv S$.  Thus, in the
nonstationary case, $a$ has a fixed sign.  Setting
$\kappa=(p-1)/p$ and integrating \eqref{eq:index-one-amplitude}, we find
some $C_*\neq0$ such that
\begin{equation}\label{eq:index-one-a-asymptotics}
a(s)=C_*e^{\kappa s}+o(e^{\kappa s})
\qquad\text{as }s\to-\infty.
\end{equation}
Indeed, \eqref{eq:exp_decay} makes the error in
$(\log|a|)'=\kappa+O(|a|)$ integrable on $(-\infty,s_0]$.

By the projected equations on $\mathscr X_0$ and $\mathscr X_+$,
\eqref{eq:index-one-dominance}, \eqref{eq:index-one-a-asymptotics}, and the
positive spectral gap on $\mathscr X_+$, we have
\[
\|g_0(\cdot,s)\|+\|g_+(\cdot,s)\|=O(e^{2\kappa s}).
\]
The parabolic estimates used above upgrade this estimate to the relative
$C^2$ topology.  Consequently,
\begin{equation}\label{eq:index-one-g-asymptotics}
g(x,s)=C_*e^{\kappa s}+o(e^{\kappa s})
\end{equation}
uniformly in the relative $C^2$ sense.

It remains to eliminate the time-translation mode.  Since
$e^{\kappa s}=(-t)^{-1}$, it follows from
\eqref{eq:index-one-g-asymptotics} that
\[
u(x,t)=U_0(x,t)\left(1+\frac{C_*}{-t}+o(|t|^{-1})\right),
\qquad
U_0(x,t)=\left[\frac{p-1}{p}(-t)\right]^{\frac1{p-1}}S(x).
\]
For $T\in\mathbb R$, let
\[
U_T(x,t)=\left[\frac{p-1}{p}(T-t)\right]^{\frac1{p-1}}S(x),
\qquad t<T.
\]
Then, uniformly in the relative topology,
\begin{equation}\label{eq:index-one-comparison-expansion}
\frac{u(x,t)-U_T(x,t)}{U_0(x,t)}
=\frac{C_*-T/(p-1)}{-t}+o(|t|^{-1})
\qquad\text{as }t\to-\infty.
\end{equation}
If $C_*>0$, choose $0<T<(p-1)C_*$.  By
\eqref{eq:index-one-comparison-expansion}, we have
$U_T(\cdot,t_0)\leq u(\cdot,t_0)$ for some sufficiently negative $t_0$.
By the comparison principle up to the extinction time $0$ of $u$, we have
$0<U_T(\cdot,0)\leq u(\cdot,0)=0$ in $\Omega$, a contradiction.
If $C_*<0$, choose
$(p-1)C_*<T<0$.  For some sufficiently negative $t_0<T$, we then have
$u(\cdot,t_0)\leq U_T(\cdot,t_0)$.  By comparison up to the extinction time
$T$ of $U_T$, we have
$0<u(\cdot,T)\leq U_T(\cdot,T)=0$ in $\Omega$, again a
contradiction.  Therefore $g\equiv0$, and under condition \textup{(i)} we obtain the
claimed separable solution.

We finally assume condition \textup{(ii)}.  In the Sobolev-subcritical
range there is a positive solution $v_\infty$ to
\eqref{eq:elliptic-equation} such that
\[
v(\cdot,s)\longrightarrow v_\infty
\quad\text{in }C^2(\overline\Omega)\quad\text{as }s\to+\infty;
\]
see \cite{JX23a}.  By the equal-energy assumption, we have
$F[v_\infty]=F[v_{-\infty}]$.  Since $F[v(\cdot,s)]$ is nonincreasing,
it is constant, and it follows from \eqref{eq:energy-derivative-identity} that
$\partial_s v\equiv0$.  This proves the conclusion under condition
\textup{(ii)} and completes the proof.
\end{proof}

\section*{Declaration}
\noindent\textbf{Data availability:} Data availability is not applicable to this article as no new data were created or analyzed in this study.

\noindent\textbf{Conflict of interest:} The authors declare that they have no conflicts of interests.

\end{document}